\documentclass[12pt]{amsart}
\usepackage{amsfonts}
\usepackage{amsmath}
\usepackage{amsxtra}
\usepackage{amssymb,latexsym}
\usepackage[mathcal]{eucal}
\usepackage{amscd}

\makeindex

\newtheorem{theo}{{\bfseries Theorem}}[section]
\newtheorem{prop}[theo]{{\bfseries Proposition}}
\newtheorem{lem}[theo]{{\bfseries Lemma}}
\newtheorem{cor}[theo]{{\bfseries Corollary}}
\newtheorem{df}[theo]{{\bfseries Definition}}

\newtheorem{ex}[theo]{{\bfseries Example}}

\def \N {\mathbb N}

\def \R {\mathbb R}

\def \A {\mathcal A}

\def \CC {\mathcal C}

\def \H {\mathcal H}

\def \M {\mathcal M}

\def \G {\mathcal G}

\def \a {\alpha }
\def \b {\beta}

\def \ep {\epsilon}

\def \d {\delta}
\def \g {\gamma}
\def \r {\rho}
\def \s {\sigma}

\def \tto {\longrightarrow}

\usepackage{amssymb,latexsym}
\usepackage[mathcal]{eucal}
\usepackage{amscd}
\usepackage{amsmath}
\usepackage{graphicx}
\usepackage{amsfonts}
\usepackage{amscd}

\numberwithin{equation}{section}

\begin{document}

\title[Mimicking an Arbitrary Tournament]{\bfseries  Generalized Intransitive Dice: \\ Mimicking an Arbitrary Tournament}
\vspace{1cm}
\author{Ethan Akin}
\address{Mathematics Department \\
    The City College \\ 137 Street and Convent Avenue \\
       New York City, NY 10031, USA     }
\email{ethanakin@earthlink.net}

\date{January 2019, revised April, 2019}

\begin{abstract} A generalized $N$-sided die is a random variable $D$ on a sample space of $N$ equally likely outcomes taking values in
the set of positive integers. We say of independent $N$ sided dice $D_i, D_j$ that $D_i$ beats $D_j$, written
$D_i \to D_j$, if $Prob(D_i > D_j) > \frac{1}{2} $.
Examples are known of intransitive $6$-sided dice, i.e. $D_1 \to D_2 \to D_3$ but $D_3 \to D_1$. A tournament of size $n$ is a choice of direction
$i \to j$ for each edge of the complete graph on $n$ vertices. We show that if $R$ is tournament on the set $\{ 1, \dots, n \}$,
then for sufficiently large
$N$ there exist sets of independent $N$-sided dice $\{ D_1, \dots, D_n \}$ such that $D_i \to D_j$ if and only if $i \to j$ in $R$.
\end{abstract}


\vspace{1cm}

\vspace{.5cm} \maketitle

\tableofcontents

\newpage
\section{Intransitive Dice}\label{sec1}
\vspace{.5cm}

A generalized die is a cube with each face labeled with a positive number. The possibility of repeated labels is allowed.
On the standard die each of numbers $1,2, \dots 6$ occurs once. Of two dice $A_1$ and $A_2$
we say that $A_1$ beats $A_2$ (written $A_1 \to A_2$) if, when they are rolled, the probability that
$D_1 > D_2$ is greater than $\frac{1}{2}$
where  $D_1$ and $D_2$ are the independent random variables of the values displayed by the dice $A_1$ and $A_2$, respectively.
There exist  examples of nontransitive dice, or intransitive dice,
three dice $A_1, A_2, A_3$ such that $A_1 \to A_2, A_2 \to A_3,$ and $A_3 \to A_1$. For example, if we let
\begin{align}\label{exeq1}
 \begin{split}
A_1 \ &=  \ \{  3, 5, 7 \}, \\
A_2 \ &=  \ \{ 2, 4, 9 \}, \\
A_3 \ &=  \ \{  1, 6, 8 \}.
 \end{split}
 \end{align}
and repeat each label twice to get $6$-sided dice, then $P(D_{i} > D_{i+1}) = \frac{5}{9}$ for $i=1,2,3$ (counting mod $3$).

The Wikipedia page on
\emph{Nontransitive Dice} contains a lovely exposition with a number
of different examples constructed by Efron, Grime and others.

On a sample space of $N$ equally likely outcomes, which we will call the \emph{faces}, an $N$-sided die\index{dice!$N$-sided} is a random variable
taking positive integer values.
Such a die is called \emph{proper}\index{dice!proper}\index{proper dice} when it takes values in the set $\{1, \dots, N \}$ and the sum of the values is $N(N+1)/2$, or,
equivalently, when the expected value of a roll is $(N+1)/2$. That is, the sum is the same as that of the
\emph{standard} $N$-sided die \index{standard $N$-sided die} with each value among  $\{1, \dots, N \}$ occurring once.
If we repeat the labels of (\ref{exeq1}) three times each then we obtain proper $9$-sided dice with a cyclic pattern.

These have been considered by Gowers in his blog and by Corey et al \cite{CG}.
Considering large numbers of such dice leads us to the theory of
tournaments.

A \emph{digraph}\index{digraph} on a nonempty set $I$ is a subset $R \subset I \times I$ such that
$R \cap R^{-1} = \emptyset$ with $R^{-1} = \{ (j,i) : (i,j) \in R \}$.
In particular, $R$ is disjoint from
the diagonal $\Delta =\{(i,i) : i \in I\}$.
 We write $i \to j$ for $(i,j) \in R$.
For $i \in S,$ the \emph{output set} $ R(i) = \{ j : (i,j) \in R \}$ and so $R^{-1}(i) = \{ j : (j,i) \in R \}$
is the \emph{input set}. \index{output set}\index{input set}
If $J \subset I$, then the \emph{restriction}\index{restriction}\index{digraph!restriction}
of $R$ to $J$ is $R|S_0 = R \cap (J \times J)$.

A digraph $R$ on $I$ is called a \emph{tournament}\index{tournament} when
$R \cup R^{-1} = (I \times I) \setminus \Delta$.
Thus, $R$ is a tournament on $S$ when for each pair of distinct elements $i, j \in S$
either $(i,j)$ or $(j,i)$ lies in $R$ but not both.
Harary and Moser provide a nice exposition of tournaments in \cite{HM}.

Our purpose here is to show that any tournament can be mimicked by a suitable choice of $N$ sided dice. That is,
with $[n] = \{ 1,2, \dots, n \}$ \index{$[n]$} we prove the following.

\begin{theo}\label{maintheo} If $R$ is a tournament on $[n]$, then
there is a positive integer $M$ such that for every integer $N \geq M$,
there exists a set
$D_1, \dots, D_n$ of independent, proper $N$-sided dice
such that for $i, j \in [n]$, $D_i \to D_j$ if and only if $i \to j$ in $R$.  That is, for $i, j \in [n]$,
\begin{equation}\label{maineq}
P(D_i > D_j) > \frac{1}{2} \quad \Longleftrightarrow \quad (i,j) \in R,
\end{equation}
\end{theo}
\vspace{.5cm}

Since the number of tournaments on $[n]$ is finite ($=  2^{n(n-1)/2}$), we may choose $M$ large enough that
for every $N \geq M$, every tournament on $[n]$ can be mimicked by proper $N$-sided dice.

In  Corey et al \cite{CG} the authors' numerical work led them to the much stronger conjecture that, with $n$ fixed, and
letting $N$ tend to
infinity, every tournament on $[n]$ becomes equally likely.

We will call $X$ a \emph{continuous random variable on the unit interval}\index{continuous random variable}
 when the distribution function $F$ of $X$ is  strictly increasing and continuous on the unit interval $[0,1]$
 with $F(0) = 0$ and $F(1) = 1$. Equivalently, the associated measure is nonatomic with support equal to $[0,1]$.
To prove Theorem \ref{maintheo} we replace the discrete random variables given by the dice with
continuous random variables on the unit interval.

For an independent pair of such random variables, say that $X$ beats $Y$, written
$X \to Y$,  if $P(X > Y) > \frac{1}{2}$. Notice that
for independent continuous random variables the probability that $X = Y$ is zero.

 We will call $X$, a
continuous random variable on the unit interval, \emph{proper}\index{proper continuous random variable}\index{continuous random variable!proper},
when the expected value, $E(X)$, is equal to $ \frac{1}{2}$. This is the
analogue of the proper condition for dice.

In the following sections we will prove:

\begin{theo}\label{maintheo2} If $R$ is a tournament on $[n]$, then exists  a set
$X_1, \dots, X_n$ of independent, proper, continuous random variables on the unit interval
such that for $i, j \in [n]$, $X_i \to X_j$ if and only if $i \to j$ in $R$.  That is, for $i, j \in [n]$,
\begin{equation}\label{maineq2}
P(X_i > X_j) > \frac{1}{2} \quad \Longleftrightarrow \quad (i,j) \in R,
\end{equation}
\end{theo}
\vspace{.5cm}

We complete this section by showing how Theorem \ref{maintheo} follows from Theorem \ref{maintheo2}.
That is, we obtain the discrete result from the continuous one.

\begin{lem}\label{discretelem}For a continuous $[0,1]$ valued random variable $X$ and  $\ep > 0$,
there exists a random variable $Y$
 with finitely many values, all rational
in $[0,1)$, with rational probabilities for each value and such that such that $P(|X - Y| > \ep) < \ep$.
In addition, if $E(X)$ is rational, $Y$ can be chosen so that
$E(Y) = E(X)$. \end{lem}

\begin{proof} For $x \in \R$ let
$\lfloor x \rfloor$ be the largest integer less than or equal to $x$.
We may assume $\ep < 1$. Choose a positive integer $M$ with $M \ep > 2$,  and let
$Y_1 = \frac{\lfloor M \cdot X \rfloor }{M}$ for $X \not= 1$ and $Y_1 = \frac{M-1}{M}$ if $X = 1$. Hence,
$ Y_1 \leq \min(X, \frac{M-1}{M})$. In addition, $X - Y_1   \leq \frac{1}{M}$.
Furthermore, the inequality is strict unless $X = 1$ and
since $X$ is continuous, $P(X = 1) = 0$. It follows that
$$E(X) - \frac{1}{M} \ < \ E(Y_1) \ \leq  \ \min(E(X), \frac{M-1}{M}).$$

For each $k = 1,\dots, M-1$ we can move some weight
from $k/M$ to $0$
to obtain $Y_2 \leq Y_1$ with rational probabilities for each value. Technically, $Y_2 = I \cdot Y_1$ where $I$ is a
suitably chosen Bernoulli random variable.
We can make the total weight change, i.e. $P(I = 0)$, arbitrarily
small so that  $P(Y_1 - Y_2 > 0 ) < \d $ with $\d < \ep$ and $\d < E(Y_1) - E(X) + \frac{1}{M}$.
Hence, we have
$$E(X) - \frac{1}{M} \ < \ E(Y_2) \ \leq  \ \min(E(X), \frac{M-1}{M}).$$ If $E(X)$ is
irrational, let $Y = Y_2$. If $E(X)$ is rational, then let
 $Y = Y_2 + E(X) - E(Y_2)$. Since $0 \leq E(X) - E(Y_2) < \frac{1}{M}$ it follows
 that $Y$ has values in $[0,1)$.  Finally,
 $$P(|X - Y| > \frac{2}{M}) < P(Y_1 - Y_2 > 0 ) < \ep.$$

\end{proof}
\vspace{.5cm}

First, we construct a set of dice with a given tournament.

\begin{proof} We obtain a set of dice which mimics tournament $R$  by approximating the sequence $\{ X_1, \dots, X_n \}$ from
Theorem \ref{maintheo2}.

There exists a positive $\g$ such that $P(X_i - X_j > \g) > \frac{1 }{2} + \g$ for all $(i,j) \in R$.

Now with $\ep = \g /3$ use Lemma \ref{discretelem} to choose finite, $[0,1)$-valued,
rational $Y_j$ so that $P(|X_j - Y_j| > \ep) < \ep$ for $j \in [n]$
and such that $E(Y_j) = \frac{1}{2}$.

Let $N$ be a common denominator
for all of the values and probabilities. Thus, for $j \in [n]$ and $k = 0, \dots, N-1$ there exists a non-negative integer
$P_{jk}$ such that
\begin{equation}
P(Y_j = \frac{k}{N}) = \frac{P_{jk}}{N}.
\end{equation}
Since the probabilities sum to one, and the expected value is $1/2$ we have for each $j \in [n]$
\begin{equation}\label{summeq}
\begin{split}
\sum_{k=0}^{N-1} \ P_{jk} \ = \ N, \\
\sum_{k=0}^{N-1} \ k P_{jk} \ = \ \frac{N^2}{2}.
\end{split}
\end{equation}
In particular, we see that $N$ must be even.

 Define an $N$-sided die $\bar D_j$
so that on $P_{jk}$ of the faces, the value $k+1$ is displayed. The outcome $\bar Z_j$
of a roll of the die $\bar D_j$ has the distribution  of $N Y_j + 1$.
It follows that for  $(i,j) \in R$
\begin{align}\label{summeq2}
\begin{split}
P(\bar Z_i - \bar Z_j > 0) \ = \ &P(Y_i - Y_j > 0) \ = \\ P(X_i - X_j + (Y_i - X_i) + (&X_j - Y_j) > 0) \ \geq \\
P(X_i - X_j > \g) - P(|Y_i - X_i|  &> \g /3 ) - P(|X_j - Y_j|  > \g /3 )\\  >   \frac{1}{2} + \g/3. \qquad &
\end{split}
\end{align}

Thus, $\bar D_i \to \bar D_j$ if and only if $(i,j) \in R$.

The sum of the face values is:
\begin{equation}\label{summeq2a}
\sum_{k=0}^{N-1} \ (k + 1) P_{jk} \ = \ \frac{N(N+2)}{2}.
\end{equation}
So these are not quite proper dice.

In order to obtain proper dice, we need some additional work. First, we can choose the
common denominator arbitrarily large. We will require that
\begin{equation}\label{summeq3}
(\frac{1}{2} + \frac{\g}{3}) \cdot (\frac{N}{N+1})^2 > \frac{1}{2}.
\end{equation}

The die $D_j$ will be $N+1$ sided. On $P_{jk}$ of the faces, the value $k+1$ is
displayed as before. In addition there is one new face
with the value $\frac{N}{2} + 1 = \frac{N + 2}{2} $ (Recall that $N$ is even). It follows that the sum of the values is
\begin{equation}\label{summeq4}
\frac{N(N+2)}{2} + \frac{N + 2}{2} \ = \ \frac{(N+1)(N+2)}{2},
\end{equation}
and so these are proper dice.

For each $j \in [n]$ let $I_j$ be the indicator of the event that when $D_j$ is rolled, one of the old faces turns up.
Thus, $I_j$ is a Bernoulli random variable
with $P(I_j = 1) = \frac{N}{N+1} $. If $Z_j$ is the outcome of a roll of $D_j$ then
\begin{equation}\label{summeq5}
Z_j \ = \ I_j \cdot (\bar Z_j) + (1 - I_j) \cdot (\frac{N}{2} + 1).
\end{equation}
Conditioned on the assumptions that $I_i = 1$ and $I_j = 1$, $Z_i > Z_j$ if and only if $\bar Z_i > \bar Z_j$.
Hence, for  $(i,j) \in R$
\begin{equation}\label{summeq6}
\begin{split}
P( Z_i -  Z_j > 0) \ \geq \ P(\bar Z_i - \bar Z_j > 0)\cdot P(I_i = 1, I_j = 1) \ =  \\
P(\bar Z_i - \bar Z_j > 0)\cdot (\frac{N}{N+1})^2 \ > \ (\frac{1}{2} + \g/3) \cdot (\frac{N}{N+1})^2 \ > \ \frac{1}{2}.
\end{split}
\end{equation}

Thus, $\{ D_1, \dots, D_n \}$ is the required list of proper $N+1$-sided dice.

\end{proof}
\vspace{.5cm}

To complete the proof of Theorem \ref{maintheo} we show that the tournament can be mimicked by $N$ sided dice for
sufficiently large $N$.

Assume that $D$ is a proper $N$-sided die such that there are $P_k$ faces with value $k$ for $k = 1, \dots , N$. For $M$ a positive integer and
$S$ an integer with $0 \leq S < N$, we define the $MN + S$ sided extension $\hat D$ such that
\begin{align}\label{extenddieeq}
\begin{split}
\text{ For } \ \ Q = 1, \dots, M, \quad &\text{there are} \ P_k \\ \text{faces with value}& \ (Q-1)N + k  \ \text{for} \ k = 1, \dots , N, \\
\text{ For } \ \ i = 1, \dots, S, \quad &\text{there is one face with value  } \ \ MN + i.
\end{split}
\end{align}
It is clear that with $S = 0$ the face sum is $M \cdot \frac{N(N+1)}{2}  \ + \ N^2 \cdot \frac{M(M-1)}{2}$ and this equals
$\frac{MN(MN+1)}{2}$. As the additional faces are added, the die remains proper. To see this, note that if the $MN$ die had been
standard, it would have remained standard as the additional faces are added.

\begin{lem}\label{extenddielem} Assume that $R$ is a tournament on $[n]$ and that $D_1, \dots, D_n$ is a list of proper $N$-sided
dice such that for some $\ep > 0$
\begin{equation}\label{extenddieeq2}
(i,j) \in R \quad \Longrightarrow \quad P(D_i > D_j) \ > \ \frac{1}{2} + \ep.
\end{equation}
 If $M$ is large enough that $2MN \ep > 1$ then for all $S = 0, \dots, N-1$ the $MN + S$ extensions $\hat D_1, \dots, \hat D_n$ are proper
 dice satisfying
  \begin{equation}\label{extenddieeq3}
(i,j) \in R \quad \Longrightarrow \quad P(\hat D_i > \hat D_j) \ > \ \frac{1}{2}.
\end{equation}
\end{lem}

\begin{proof} For $Q = 1, \dots, M$, think of the values between $(Q-1)N + 1$ and $QN$ as the $Q^{th}$ full block of values
and the values $MN + 1, \dots MN + S$ as
the partial block. With $(i,j) \in R$ we condition on the following cases:
\begin{itemize}
\item Assuming $\hat D_i$ and $\hat D_j$ occur in different blocks, then
$P(\hat D_i > \hat D_j) = \frac{1}{2}$, because ties cannot occur and it is equally likely that
$\hat D_i$ or $\hat D_j$ occurs in a higher block.

\item  Assuming $\hat D_i$ and $\hat D_j$ occur in the same full block, then $P(\hat D_i > \hat D_j) > \frac{1}{2} + \ep$. The probability that
they occur in the same full block is $\frac{MN^2}{(MN + S)^2}$.

\item Assuming that $\hat D_i$ and $\hat D_j$ both occur in the partial block, $P(\hat D_i = \hat D_j) = \frac{1}{S}$ and so
$P(\hat D_i > \hat D_j) = \frac{1}{2} (1 - \frac{1}{S})$. The probability that
they both occur in the partial block is $\frac{S^2}{(MN + S)^2}$.
\end{itemize}

It follows that for $(i,j) \in R$,  $P(\hat D_i > \hat D_j)$ is at least $ \frac{1}{2}$ plus the deviation
  \begin{equation}\label{extenddieeq4}
  \ep \cdot \frac{MN^2}{(MN + S)^2} - \frac{1}{2S} \cdot \frac{S^2}{(MN + S)^2}.
  \end{equation}
  Since $S < N$, this deviation is positive when $2MN \ep > 1$.

\end{proof}

{\bfseries Remark:} Observe that for any positive integer $M$, if $S = 0$, then for $(i,j) \in R$,  $P(\hat D_i > \hat D_j)$ is at least $ \frac{1}{2}$
plus the deviation $\ep/M$.
With $M = 2$ and $S = 0$ we call the $2N$ extension $\hat D$ of an $N$-sided die $D$  the \emph{double}\index{dice!double} of $D$.
\vspace{1cm}

\section{Homeomorphism Groups}
\vspace{.5cm}

Let $\H$ \index{$\H$}\index{homeomorphism group!$\H$}denote the group of orientation preserving homeomorphisms on
$[0,1]$. Thus, $F \in \H$ when it is a strictly increasing, continuous real-valued function on $[0,1]$ with
$F(0) = 0$ and $F(1) = 1$. Let  $I$ be the identity element so that $I(x) = x$ for $x \in [0,1]$.
For $F \in \H$ let $\r_F$ be the right translation map on $\H$ given by  $\r_F(G) = G \circ F$.

$\H$ is a subset of the Banach space $\CC([0,1])$ of
continuous real-valued functions on $[0,1]$ and, when equipped with the metric
induced by the sup norm $|| \cdot ||$, $\H$ is a topological group via composition.

We call $X$ a continuous random variable on $[0,1]$ exactly when its the distribution function \index{distribution function} $F_X$ is
an element of the group $\H$. Technically, the distribution function is defined on $\R$ and is zero below $0$ and
one above $1$, but we will restrict to $[0,1]$.

Let $U$ be a uniform random variable on $[0,1]$, i.e. $U \sim Unif(0,1)$, \index{$Unif(0,1)$} so that $F_U = I$.
If $F \in \H$, then the random variable $X = F^{-1}(U)$ has distribution function $F$.
That is, for $x \in [0,1]$ $P(X < x) = P( U < F(x) ) = F(x)$, see, e.g.
\cite{BH} Chapter 5 on the Universality of the Uniform.
Notice that for the expected value of such a random variable, we can integrate by parts to get
\begin{align}\label{expecteq}
\begin{split}
E(X) \ = \ &\int_0^1  x \ dF(x) \ = \\
xF(x)|_0^1 \ - \ \int_0^1 & F(x) \ dx  \ = \ 1 \ - \ \int_0^1 F(x)\ dx.
\end{split}\end{align}

Now suppose that for $i \in [n], \ X_i = F_i^{-1}(U_i)$ where the $U_i$'s are independent $Unif(0,1)$ random variables
and $F_i \in \H$.

\begin{align}\label{01}
\begin{split}
P(X_i > X_j) \ = \ &P(F_i^{-1}(U_i) > F_j^{-1}(U_j)) \\ = \ P(F_j(F_i^{-1}&(U_i)) > U_j).
\end{split}\end{align}
Conditioning on the assumption $U_i = x$, this is $$P(F_j(F_i^{-1}(x)) > U_j) = F_j(F_i^{-1}(x)).$$
Since $U_i$ is uniform, it follows that
 \begin{equation}\label{02}
P(X_i > X_j) \ = \ \int_0^1  F_j(F_i^{-1}(x)) \ dx .
\end{equation}

Each $F \in \H$ is a strictly increasing function from $[0,1]$ onto
$[0,1]$. Note that the integral $\int_0^1  F^{-1}(x) \ dx$ is the area in
the square $[0,1] \times [0,1]$ between the y-axis and the set
$$\{ (F^{-1}(y),y) : 0 \leq y \leq 1 \} \ = \ \{ (x,F(x)) : 0 \leq x \leq 1 \} $$
which is the complement in the square of the region under the graph of $F$.  Thus,
\begin{equation}\label{sumeq}
\int_0^1  F(x) + F^{-1}(x) \ dx = 1
 \end{equation}
for all $F \in \H$.

Define a partition of $\H$ by
\begin{equation}\label{setdefeq}
\begin{split}
\H_+ \ = \ \{ F \in \H : \int_0^1  F(x)  \ dx > \frac{1}{2} \} \\
\H_- \ = \ \{ F \in \H : \int_0^1  F(x)  \ dx < \frac{1}{2} \} \\
\H_0 \ = \ \{ F \in \H : \int_0^1  F(x)  \ dx = \frac{1}{2} \} \\
 \end{split}
 \end{equation}
 \index{$\H_+$} \index{$\H_-$} \index{$\H_0$}

 The sets $\H_+$ and $\H_-$ are open and with  union dense in $\H$, because if
 $F \in \H_0$ and $G \in \H$ with $G \geq F$ and $G \not= F$, then $G \in \H_+$.
 Similarly $F$ can be perturbed to an element of $\H_-$.
From (\ref{sumeq})
 we see that $F \in \H_+$ if and only if $F^{-1} \in \H_-$.

 If $X = F^{-1}(U)$ with $U \sim Unif(0,1)$ and $F \in \H$, it follows from (\ref{expecteq}) that
 \begin{equation}\label{expecteq2}
 E(X) \ = \ \frac{1}{2} \qquad \Longleftrightarrow \qquad F \in \H_0.
 \end{equation}

  Define on $\H$ the digraph $\Gamma_{\H}$ \index{$\Gamma_{\H}$}\index{digraph!$\Gamma_{\H}$} by
 \begin{equation}\label{gamedefeq}
 (F,G) \in \Gamma_{\H} \ \Longleftrightarrow \ F \to G \ \Longleftrightarrow \ G \circ F^{-1} \in \H_+.
 \end{equation}
 Thus, $I \to G$ if and only if $G \in \H_+$ and $G\to I$ if and only if $G \in \H_-$. The digraph is invariant with respect to right
 translation. Consequently, for every $F \in \H$, the union
 \begin{equation}
 \Gamma_{\H}(F)  \cup \Gamma_{\H}^{-1}(F) = \r_F(\H_+) \cup \r_F(\H_-)
  \end{equation}
  is open and dense in $\H$.

From (\ref{02}) we see that $(F_i,F_j)  \in \Gamma_{\H}$ if and only if $X_i \to X_j$ when
$X_i$ and $X_j$ are independent random variables with distribution functions $F_i$ and $F_j$, respectively.
Thus, Theorem \ref{maintheo2} is equivalent to the statement that every finite tournament can be embedded in the restriction to $\H_0$
of the $\Gamma_{\H}$ digraph on $\H$.

To prove this, it is convenient to  shift the interval from $[0,1]$ to $[-1,1]$.  Let
$\G$ \index{$\G$}\index{homeomorphism group!$\G$} denote the group of
orientation preserving homeomorphisms
on $[-1,1]$ and let $i$ be the identity element so that $i(t) = t$ for $t \in [-1,1]$.
Let $\CC([-1,1])$ denote the separable Banach space of continuous,
real-valued functions on $[-1,1]$ so that the subset $\G$ is a topological group. On $\G$ the distance
$d(f,g) = \max( || f - g||, ||f^{-1} - g^{-1}|| )$ defines a complete metric, topologically equivalent to
one induced by the sup norm.

 Define $q: [-1,1] \tto [0,1]$ by $x = q(t) = \frac{t + 1}{2}$ so that $ t = 2x - 1$.
 The conjugation map $A_q$ given by $f = A_q(F) = q^{-1}\circ F \circ q$ is a topological group isomorphism from
 $\H$ to $\G$, with $F(x) = \frac{f(2x - 1) +1}{2}$.

It follows that
 \begin{equation}\label{inteq}
 \int_0^1  F(x) \ dx \ = \ \frac{1}{2} \  + \ \frac{1}{4} \int_{-1}^{1}  f(t) \ dt.
 \end{equation}
 In particular, from (\ref{sumeq}) we obtain
 \begin{equation}\label{sumeq2}
\int_{-1}^1 f(t) + f^{-1}(t) \ dt = 0.
 \end{equation}

Define the partition of $\G$ by
\begin{equation}\label{setdefeq2}
\begin{split}
\G_+ \ = \ \{ f \in \G : \int_{-1}^1 f(t)  \ dt > 0 \}, \\
\G_- \ = \ \{ f \in \G : \int_{-1}^1 f(t)  \ dt < 0 \}, \\
\G_0 \ = \ \{ f \in \G : \int_{-1}^1 f(t)  \ dt = 0 \},  \\
 \end{split}
 \end{equation}
\index{$\G_+$} \index{$\G_-$} \index{$\G_0$}
We see that $A_q$ maps $\H_+, \H_-$ and $\H_0$ to $\G_+, \G_-$ and $\G_0$, respectively.
 Since it is a group isomorphism, $A_q$ provides an
 isomorphism between the digraphs $\Gamma_{\H}$ and  $\Gamma_{\G}$ with
 $\Gamma_{\G} = \{ (f,g) \in \G \times \G : g \circ f^{-1} \in \G_+ \}$. That is,
  \begin{equation}\label{gamedefeq2}
 (f,g) \in \Gamma_{\G} \ \Longleftrightarrow \ f \to g \ \Longleftrightarrow \ \int_{-1}^1 g(f^{-1}(t)) \ dt  >  0.
 \end{equation}
 \index{$\Gamma_{\G}$}\index{digraph!$\Gamma_{\G}$}

 Thus, Theorem \ref{maintheo2} is equivalent to the following result
 which we will prove in the next section.

 \begin{theo}\label{maintheo3} Every finite tournament can be embedded in the restriction of the $\Gamma_{\G}$ digraph to $\G_0$. \end{theo}
 \vspace{.5cm}

 We let $\bar \G$  \index{$\bar \G $} denote the set of continuous, non-decreasing maps from $[-1,1]$ onto itself.
 So $f \in \bar \G$ when $f \in \CC([-1,1])$,
 $f(\pm 1) = \pm 1$, and for all $t_1, t_2 \in [-1,1], \ t_1 < t_2$ implies $f(t_1) \leq f(t_2)$. Since composition is  jointly continuous,
 $\bar \G$ is a topological semigroup with $\G$ the group of invertible elements. Let  \index{$\bar \G_{0} $}
 $$ \bar \G_0 \ = \ \{ f \in \bar \G : \int_{-1}^1 f(t)  \ dt = 0 \}.$$

 For $f \in \bar \G$ we let $\r_f$ \index{$\r_f$}  denote the right translation map \index{translation map!right} on
 $\bar \G$ so that $\r_f(g) = g \circ f$.
 If $f \in \G$ then $\r_f$ is a homeomorphism with inverse $\r_{f^{-1}}$ and it maps $\G$ to itself. As before, the digraph
 $\Gamma_{\G}$ is $\r_f$ invariant for all $f \in \G$.

 Define $z : \R \tto \R$ by $z(t) = -t$.  Of course, $z([-1,1]) = [-1,1]$. The conjugation map $A_z$ given by $A_z(f) = z\circ f \circ z$
 is a linear isometry on $\CC([-1,1])$, which restricts to a
  a topological semigroup isomorphism from
 $\bar \G$ to itself which preserves $\G$. For $f \in \CC([-1,1])$ we write $f^*$ \index{$f^*$}
 for $A_z(f)$ so that $f^*(t) = - f(-t)$. Using the substitution
 $s = -t$ we see that
 \begin{equation}\label{inteq1}
 \int_{-1}^1 f^*(t) \ dt \ = \ -  \int_{-1}^{1} f(s) \ ds.
 \end{equation}
 Clearly, $f^{**} = f$.

 Since $A_z$ is a topological semigroup isomorphism and is a group isomorphism on $\G$ it is clear that
  \begin{align}\label{inteq1a}
  \begin{split}
  (f \circ g)^* \ = \ f^* \circ g^* & \qquad \text{for} \ f,g \in \bar \G,\\
  (f^*)^{-1} \ = \ (f^{-1})^* & \qquad \text{for} \ f \in \G.
\end{split}
\end{align}

Let $ \CC_{00} = \{ f \in  \CC([-1,1]): f = f^* \} $  and let $\bar \G_{00}$ and $ \G_{00}$ equal $\bar \G \cap \CC_{00}$ and
$\G \cap \CC_{00}$. \index{$\CC_{00}$}\index{$\bar \G_{00}$} \index{$\G_{00}$}
Thus, $\CC_{00}$,
$\bar \G_{00}$ and $\G_{00}$ consist of the odd functions in $\CC([-1,1])$, $\bar \G$ and $\G$, respectively.
From (\ref{inteq1a}) it is clear that
$\bar \G_{00}$ is a closed subsemigroup of $\bar \G$
that $\G_{00}$ is a subgroup of $\G$.

Clearly,
$$ \bar \G_{00} \ \subset \ \bar \G_0 \quad \text{and} \quad \G_{00} \ \subset \ \G_0.$$

We collect some elementary results.

\begin{prop}\label{elemgrpprop} \begin{enumerate}
\item[(i)] $\bar \G, \bar \G_0, \CC_{00}, \bar \G_{00}, \G, \G_0$ and $ \G_{00}$
are convex subsets of $\CC([-1,1])$. Moreover,
$(x_1f_1 + x_2f_2)^* = x_1 f_1^* + x_2 f_2^*$ for all $f_1, f_2 \in \CC([-1,1])$ and $x_1, x_2 \in \R$.

\item[(ii)] $\bar \G, \bar \G_0 $ and $ \bar \G_{00}$ are the closures in $\CC([-1,1])$
of $ \G, \G_0$ and $ \G_{00}$, respectively.
Moreover, $\G$ is a dense $G_{\delta}$ subset of $\bar \G$,  and $\bar \G \setminus  \bar \G_0 $ is
a dense, relatively open subset of $\bar \G$.

\item[(iii)] Each of the sets $\bar \G, \bar \G_0, \bar \G_{00}, \G, \G_0$ and $ \G_{00}$ is
preserved by $A_z$.  Furthermore, the map
$f \mapsto \frac{1}{2}(f + f^*)$ is a retraction of $\CC([-1,1])$ onto $\CC_{00}$,
taking $\bar \G$ onto $\bar \G_{00}$, and $\G$ onto $\G_{00}$.

\item[(iv)] If $f_1, \dots, f_n \in \G$, then $f_M, f_m \in \G$ with
$$f_M(t) \ = \ \max_i \ f_i(t) \quad \text{and} \quad  f_m(t) \ = \ \min_i \ f_i(t).$$
\end{enumerate}
\end{prop}

\begin{proof} (i): Convexity of the various subsets  is obvious.

(ii): The sets $\bar \G, \bar \G_0 $ and $ \bar \G_{00}$ are clearly closed in $\CC([-1,1])$.

Furthermore, if $f_1 \in \bar \G$ and $f_2 \in \G$, then $(1-x)f_1 + xf_2 \in \G$ for all $x \in (0,1]$ and so $\G$ is dense in $\bar \G$.
Similarly, if $f_1 \in \bar \G_0$ and  $f_2 \in \bar \G \setminus \bar \G_0$,  then $(1-x)f_1 + xf_2 \in \bar \G \setminus \bar \G_0$ for all $x \in (0,1]$.
Since $\bar \G \setminus \bar \G_0$ is clearly nonempty, it is dense in $\bar \G$.

For a fixed $t_1 < t_2$ the condition $f(t_1) < f(t_2)$ is an open condition on $f$.
Intersecting on all such pairs with $t_1$ and $t_2$ rational, we obtain $\G$ as a $G_{\delta}$ subset of $\bar \G$.

(iii): $\frac{f(t) + f^*(t)}{2} = \frac{f(t) - f(-t)}{2}$ and so $\frac{1}{2}(f + f^*)$ is just the odd part of $f$.

(iv): If $t_1 < t_2$ and $f_M(t_1) = f_i(t_1)$, then $f_M(t_2) \geq f_i(t_2) > f_i(t_1)$ because $f_i \in \G$. Similarly, $f_m \in \G$.

\end{proof}

{\bfseries Remark:} The set $\G$ is not open in $\bar \G$.  In fact, it is easy to check that $\bar \G \setminus \G$ is dense in
$\bar \G$.

\vspace{.5cm}

Let $\tilde q : [-1,1] \tto [-1,0]$ by $\tilde q(t) = \frac{t - 1}{2}$.

We define the map $\odot : \bar \G \times \bar \G \to \bar \G$ so that the restriction $f_1 \odot f_2|[-1,0]$ equals
$\tilde q \circ f_1 \circ (\tilde q)^{-1}$ and $f_1 \odot f_2|[0,1]$ equals $q \circ f_2\circ (q)^{-1}$. That is,
\begin{equation}\label{compeq}
f_1 \odot f_2 (t) \ = \ \begin{cases} \frac{f_1(2t +1) - 1}{2} \quad \text{for} \ t \in [-1,0], \\
  \frac{f_2(2t -1) + 1}{2} \quad \text{for} \ t \in [0,1]. \end{cases}
  \end{equation}
\index{$f_1 \odot f_2$}
  By using the substitutions, $s = 2t + 1$ on $[-1,0]$ and $= 2t - 1$ on $[0,1]$ we obtain
  \begin{equation}\label{inteq2}
  \int_{-1}^1 f_1 \odot f_2 (t) \ dt \ = \ \frac{1}{4}[\int_{-1}^{1} f_1(s) \ ds \ + \ \int_{-1}^{1}  f_2(s) \ ds].
  \end{equation}

\begin{prop}\label{elemgroupprop2} Let $f_1, f_2, g_1, g_2 \in \bar \G$. \begin{enumerate}

\item[(i)] $(f_1 \odot f_2) \circ (g_1 \odot g_2) \ = \ (f_1 \circ g_1) \odot (f_2 \circ g_2)$. Moreover, if $f_1, f_2 \in \G$,
 then $f_1 \odot f_2 \in \G$ with $(f_1 \odot f_2)^{-1} \ = \ (f_1)^{-1} \odot (f_2)^{-1}$.

  \item[(ii)] For $x \in [0,1]$,

 \begin{equation}\label{coneq}
  \begin{split}
  (xf_1 + (1-x)g_1)\odot f_2 \ = \ x(f_1 \odot f_2) + (1-x) (g_1 \odot f_2), \\
  f_1\odot (xf_2 + (1-x)g_2)\ = \ x(f_1 \odot f_2) + (1-x) (f_1 \odot g_2).
  \end{split}
  \end{equation}

 \item[(iii)] $(f_1 \odot f_2)^* \ = \ f_2^* \odot f_1^*$.
 In particular, for $f \in \bar \G$, $f \odot f^* \in \bar \G_{00}$.

 \item[(iv)] For $f \in \bar \G$, there exist $f_1, f_2 \in \bar \G$ such that $f = f_1 \odot f_2$ if and
 only if $f(0) = 0$.

 \item[(v)] For $f \in \bar \G$, $f \in \bar \G_{00}$ if and only if there exists $f_1 \in \bar \G$ such that
 $f = f_1 \odot f_1^*$.
  \end{enumerate}
  \end{prop}

   \begin{proof}  (i):  The map $A_q$ and the analogue for $\tilde q$ are homomorphisms.

   (ii): The maps $q$ and $\tilde q$ are affine.

   (iii) For $t \in [-1,0], -t \in [0,1]$ and so
   $$- (f_1 \odot f_2)(-t) = - \frac{f_2(-2t - 1) + 1}{2} = \frac{f_2^*(2t + 1) - 1}{2} = (f_2^* \odot f_1^*)(t).$$
   Similarly, for $t \in [0,1]$.

   (iv): If $f(0) = 0$, then
    \begin{align}\label{coneq2}
  \begin{split}
   f_1 \ = \ &\tilde q^{-1} \circ (f|[-1,0]) \circ \tilde q, \\
    f_2 \ = \  &q^{-1} \circ (f|[0,1]) \circ  q.
    \end{split}
  \end{align}

  (v): If $f \in \G_{00}$, then, because $f$ is odd, $f(0) = 0$.  By (iv), $f = f_1 \odot f_2$ for some $f_1, f_2 \in \bar \G$.  Since $f^* = f$,
  (iii) implies $f_1 \odot f_2 = f_2^* \odot f_1^*$ and so $f_2 = f_1^*$. The converse is in (iii).

  \end{proof}
  \vspace{.5cm}

  \begin{cor}\label{elemgroupcor3} The relatively open set $\G_0 \setminus \G_{00}$ is dense in $\G_{0}$. \end{cor}

  \begin{proof}  If $g_1, g_2 \in \G_{00}$ are distinct, then $g_1 \odot g_2 \in \G_0 \setminus \G_{00}$. In general, if
$g_1 \in \G_0 \setminus \G_{00}$ and $g_2 \in \G_{00}$, then for $x > 0$, $x g_1 + (1-x)g_2 \in \G_0 \setminus \G_{00}$
and these approach $g_2$ as $x$ tends to $ 0$.

\end{proof}
\vspace{.5cm}

  Define
  \begin{align}\label{coneq3}
  \begin{split}
 Q(f,g) \ = \ \int_{-1}^1 f(g^{-1}&(t)) \ dt \qquad \text{for} \ f, g \in \G, \\
 \text{So that} \quad g \to f \quad & \Longleftrightarrow \quad Q(f,g) > 0.
 \end{split}
 \end{align}
 \index{$Q(f,g)$}

 From (\ref{sumeq2}) and (\ref{inteq1})
 we see that for $f, g \in \G$.
 \begin{equation}\label{coneq4}
 Q(g,f) \ = \ - Q(f,g) \ = \ Q(f^*,g^*).
 \end{equation}

 So $Q(f,g) = 0$ if $f = g$ or if $f, g \in \G_{00}$.

 Recall that $\r_h$ is right translation by $h \in \bar \G$: $\r_h(f) = f \circ h$.
  \begin{equation}\label{coneq4a}
Q(\r_h(f),\r_h(g)) \ = \ Q(f,g) \qquad \text{for} \ f, g, h \in \G.
 \end{equation}

 From Proposition \ref{elemgroupprop2} and (\ref{inteq2}) we see that for $f_1, f_2, g_1, g_2 \in \G$
  \begin{equation}\label{coneq5}
  Q(f_1 \odot f_2,g_1 \odot g_2) \ = \ \frac{1}{4} [ Q(f_1,g_1) + Q(f_2,g_2)].
  \end{equation}

  Finally, $Q$ is affine in each variable separately.  That is for \\ $f_1, f_2, f, g_1, g_2, g \in \G$ and $x \in [0,1]$
  \begin{align}\label{coneq6}
  \begin{split}
  Q(xf_1 + (1-x)f_2,g) \ &= \ xQ(f_1,g) + (1-x)Q(f_2,g), \\
  Q(f,xg_1 + (1-x)g_2) \ &= \ xQ(f,g_1) + (1-x)Q(f,g_2).
   \end{split}
 \end{align}
 This is obvious for the first variable and so, from (\ref{coneq4}), it follows for the second.

 \vspace{1cm}

 \section{Constructions and Lemmas}
\vspace{.5cm}

For $f \in \CC([-1,1])$ we define $f^e = \frac{1}{2}(f - f^*)$ \index{$f^e$}so that
 $f^e(t) = \frac{f(t) + f(-t)}{2}$. Thus, $f^e$ is the even part of $f$.
If $f \in \bar \G$, then $f^e(\pm 1) = 0$. In general, since $f^{e}(-1) = f^e(1)$,  $f^e$ is never in $\bar \G$.

Since $f^e$ is even, (\ref{inteq1}) implies that
\begin{equation}\label{inteq3}
2 \int_{0}^{1}  f^e(t) \ dt \ = \  \int_{-1}^1 f^e(t) \ dt \ = \ \int_{-1}^1 f(t) \ dt.
\end{equation}

In particular, for $f \in \bar \G$,
\begin{equation}\label{inteq4}
\int_{0}^{1}  f^e(t) \ dt \ = \ 0 \qquad \Longleftrightarrow \qquad f \in \bar \G_{0}.
\end{equation}

Clearly
\begin{equation}\label{inteq5}
f^e \ = \ 0 \qquad \Longleftrightarrow \qquad f \in  \CC_{00}.
\end{equation}

For any $x_1, x_2 \in \R$
\begin{equation}\label{inteq6}
(x_1 f_1 + x_2 f_2)^e \ = \ x_1 f_1^e + x_2 f_2^e.
\end{equation}

We will need the following step-functions.

\begin{df}\label{stepdef01} With $m$ a positive integer,  an $m$-sequence pair \index{$m$-sequence pair} \index{sequence pair}on $[-1,1]$
$$[-1 = x_0 < x_1 <  \dots < x_m = 1;-1 = y_0 < y_1 < \dots < y_{m+1} = 1]$$
has \emph{associated step function}\index{associated step function} $h : [-1,1] \tto [-1,1]$  defined by
\begin{equation}\label{stepeq01}
h(t) \ = \ \begin{cases} y_i \quad \text{for} \ x_{i-1} < t < x_i,  \ i = 1, \dots, m. \\
y_0 \quad \text{for} \ t = x_0, \\
\frac{1}{2}(y_i + y_{i+1}) \quad \text{for} \ t = x_i,  \ i = 1, \dots, m-1.\\
y_{m+1} \quad \text{for} \ t = x_m,
 \end{cases}
\end{equation}

For the $m$-sequence pair on $[0,1]$
 $$[0 < x_1 <  \dots < x_m = 1; 0 < y_1 < \dots < y_{m} < 1]$$ the
\emph{associated odd step function}\index{associated odd step function}  $h : [-1,1] \tto [-1,1]$ is the step function
associated to the $2m$-sequence pair
\begin{displaymath}
\begin{split}
[-1 = - x_m <   \dots < -x_1 < 0 < x_1 < \dots < x_m = 1; \\
-1 < -y_{m} < \dots < -y_{1} < y_1 < \dots y_m < 1].
\end{split}
\end{displaymath}
\end{df}
 \vspace{.5cm}

For a sequence pair $[x_0,\dots,x_m; y_0,\dots,y_{m+1}]$ on $[-1,1]$ we define
\begin{align}\label{stepeq02}
\begin{split}
\ell_i \ = \ x_i - x_{i-1} \ &\text{for} \ 1 = 1, \dots, m \\ \text{so that} \quad x_i = x_0 + &\sum_{k=1}^i \ell_i \ \text{for} \ 1 = 1, \dots, m.
\end{split}
\end{align}
Clearly,
\begin{equation}\label{stepeq03}
\int_{-1}^1 h(t) \ dt \ = \ \sum_{i=1}^m \ \ell_i y_i.
\end{equation}

\begin{lem}\label{steplem02} For an $m$-sequence pair $[x_0,\dots,x_m;y_0,\dots,y_{m+1}]$ on $[-1,1]$, assume that $\ep > 0$
satisfies $2 \ep, 2 \bar \ep < min_i \ \ell_i$ where $\bar \ep = \ep \frac{(1 + y_1)}{(1- y_m)}$. There is a piecewise linear
 $h_{\ep} \in \bar \G$ with $\int_{-1}^1 h_{\ep}(t) \ dt = \int_{-1}^1 h(t) \ dt $ and as $\ep$ tends to $0$ (written $\ep \leadsto 0$),
 $ h_{\ep} $ converges
 pointwise to $h$. So if  $g: [-1,1] \tto \R$  is continuous, then
\begin{equation}\label{stepeq04}
\int_{-1}^1 g(h_{\ep}(t)) \ dt \ \leadsto \ \int_{-1}^1 g(h(t)) \ dt \ = \ \sum_{i=1}^m \ \ell_i g(y_i).
\end{equation}

 If $h$ is the associated odd step function to the sequence pair $$[0, x_1,\dots,x_m = 1;0,y_1\dots,y_m, y_{m+1} = 1 ]$$ on $[0,1]$, then
 each $h_{\ep}$ is an odd function and so lies in $\bar \G_{00}$.
 \end{lem}

 \begin{proof} For $i = 1, \dots, m$ replace the jump at $x_i$ by the line segment connecting the point $(x_i - \ep,y_i)$ to $(x_i + \ep,y_{i+1})$.
 Note that at $x_i$ the line passes through the mid-point $(x_i, \frac{1}{2}(y_i + y_{i+1}))$. In the area under the graphs a
rectangle with base $\ep$ and height $y_{i+1} - y_i$ is replaced by a triangle with base $2 \ep$ and height $y_{i+1} - y_i$ and
so the integral remains unchanged.

 At $x_0 = -1$ replace the jump by the line segment connecting $(-1,-1)$ $ = (x_0,y_0)$ to $(-1 + \ep,y_1)$ and at $x_m = 1$ replace the
 jump by the line connecting $(1 - \bar \ep,y_m)$ to $(1,1) = (x_m,y_{m+1})$. In the area under the graphs a
 triangle on the left with base $\ep$ and height $y_1 - y_0 = 1 + y_1$ is removed and a triangle on the right with base $\bar \ep$
 and height $y_{m+1} - y_m = 1 - y_m$ is added. The definition of $\bar \ep$ implies that the two triangles have the
 same area.  It follows that $\int_{-1}^1 h_{\ep}(t) \ dt = \int_{-1}^1 h(t) \ dt $.

 When $h$ is an associated odd step function, it is clear that each $h_{\ep}$ is odd.  In that case, note that $\bar \ep = \ep$.

 For any $t \in [-1,1]$ it is clear that for $\ep$ sufficiently small $h_{\ep}(t) = h(t)$ and so pointwise convergence is obvious.
 Since $g \circ h_{\ep}$ then converges pointwise to $g \circ h$, the integral results of (\ref{stepeq04}) follow.

 \end{proof} \vspace{.5cm}

\begin{lem}\label{newlem1} Assume that $f \in \CC([-1,1])$ with $\int_{-1}^1 f(t) \ dt = 0$ and $f(\pm 1) = 0$. If $f$ is not identically zero, then
there exist $g_+, g_- \in \G_0$ such that $\int_{-1}^1 f(g_+(t)) \ dt > 0$ and $\int_{-1}^1 f(g_-(t)) \ dt < 0$. If $f$ is even, then we can
choose $g_+, g_- \in \G_{00}$. \end{lem}

\begin{proof} First assume that $f$ is even. Observe that if $g$ is odd, then $f \circ g$ is
even and so $\int_{-1}^1 f(g(t)) \ dt = 2 \int_{0}^1 f(g(t)) \ dt$. By assumption, $ f \not= 0$ and $ 2 \int_{0}^1 f(t) \ dt =  \int_{-1}^1 f(t) \ dt = 0.$
So there exist $a_-, a_+ \in (0,1)$ such that $f(a_+) > 0, f(a_-) < 0$.

For $a$ equal to $a_+$ or $a_-$ , let $h$ be the odd step function associated with the $1$-sequence pair $[0,1;0,a,1]$ on $[0,1]$. That is,
 \begin{equation}\label{inteq10}
 h(t) \ = \ \begin{cases}  -a  \quad \text{for} \  -1 < t < 0, \\
a \ \quad \ \text{for} \  0 < t < 1, \\ \pm 1 \quad \ \text{for} \ t = \pm 1, \\
 0 \quad \ \ \ \text{for} \ t = 0.
 \end{cases}
 \end{equation}
 So $\int_0^1 f(h(t)) \ dt = f(a)$. Let $g_{\ep} = \ep i + (1 - \ep) h_{\ep}$
 where for $\ep > 0$ sufficiently small, $h_{\ep} \in \bar \G_{00}$
are the approximating functions from Lemma \ref{steplem02}. Because
$i \in \G_{00}$, $g_{\ep} \in \G_{00}$. As $\ep \leadsto 0$,
 $\int_0^1 f(g_{\ep}(t)) \ dt \leadsto \int_0^1 f(h(t)) \ dt = f(a)$.
 So we can choose $g = g_{\ep}$ with $\ep > 0$, sufficiently small so that
$\int_{0}^1 f(g(t)) \ dt$ has the same sign as $f(a)$. \vspace{.5cm}

For $f$ not necessarily even, we can choose  $a_-, a_+ \in (-1,1) \setminus \{ 0 \}$ such
that $f(a_+) > 0, f(a_-) < 0$. Let $a$ equal $a_+$ or $a_-$
 and let $b \not= a \in (-1,1)$. Define $\ell_1$ and $\ell_2$ by the equations
 \begin{align}\label{neweq1}
 \begin{split}
 \ell_1 \ + \ \ell_2 \ &= \ 2, \\
  \ell_1 b \ + \ \ell_2 a \ &= \ 0,
 \end{split} \end{align}
 So that
 \begin{equation}\label{neweq2}
 \ell_1 \ = \ \frac{2a}{a - b}, \qquad \ell_2 \ = \ \frac{-2b}{a - b}.
 \end{equation}

{\bfseries  Case 1} $(a \in (0,1))$: Choose $b \in (-1,0)$ so that $\ell_1, \ell_2 > 0$.  Let
$x^* = \ell_1 - 1 = (a + b)/(a - b)$ so that $x^* - (-1) = \ell_1, 1 - x^* = \ell_2$.

Because $f(-1) = 0$
 \begin{equation}\label{neweq3}
 \ell_1 f(b) + \ell_2 f(a)  \  \leadsto  \ \frac{2}{a + 1} \cdot f(a).
 \end{equation}
 as $ b  \leadsto -1$.

 Let $h$ be the step function associated with the $2$-sequence pair \\ $[-1, x^*, 1; -1, b, a, 1]$ so that
 $\int_{-1}^1 h(t) \ dt = \ell_1 b +  \ell_2 a = 0$ and
 $\int_{-1}^1 f(h(t)) \ dt = \ell_1 f(b) + \ell_2 f(a)$.  From (\ref{neweq3})
 we can choose $b$ close enough to $-1$ so that
 this has the same sign as $f(a)$.

 As above, choose $g = g_{\ep} = \ep i + (1 - \ep) h_{\ep}$ with
 $\ep > 0$ sufficiently small so that $\int_{-1}^1 \ f(g(t)) \ dt$ has the
 same sign as $f(a)$. By  Lemma \ref{steplem02} each $g_{\ep} \in \G_0$ for $\ep > 0$. \vspace{.5cm}

{\bfseries  Case 2} $(a \in (-1,0))$: Choose $b \in (0,1)$ so that again $\ell_1, \ell_2 > 0$.
Let $x^* = \ell_2 - 1 = (a + b)/(b - a)$ so that $x^* - (-1) = \ell_2, 1 - x^* = \ell_1$.

Because $f(1) = 0$
 \begin{equation}\label{neweq4}
 \ell_1 f(b) + \ell_2 f(a)  \  \leadsto  \ \frac{2}{1 - a} \cdot f(a).
 \end{equation}
 as $ b  \leadsto 1$.

 Let $h$ be the step function associated with the $2$-sequence pair\\ $[-1, x^*, 1; -1, a, b, 1]$ so that
 $\int_{-1}^1 f(h(t)) \ dt = \ell_1 f(b) + \ell_2 f(a)$.  From (\ref{neweq4}) we
 can choose $b$ close enough to $1$ so that
 this has the same sign as $f(a)$.

 As above, choose $g = g_{\ep} = \ep i + (1 - \ep) h_{\ep} \in \G_0$ with  $\ep > 0$
 sufficiently small so that $\int_{-1}^1 \ f(g(t)) \ dt$ has the
 same sign as $f(a)$.

 \end{proof} \vspace{.5cm}

 Now define the sequences $\{ p_0, p_1, \dots \}$ and $\{ q_0, q_1, \dots \}$ in $\CC([-1,1])$ by:
  \begin{equation} \label{neweq5}
  p_k(t) \ = \ t^{2k+1}, \qquad q_k = p_{2k} \odot p_{2k+1},
  \end{equation}
  for $k = 0, 1, \dots$.

  \begin{lem}\label{newlem3} The sequence $\{ p_0, p_1, \dots \}$ is a linearly
  independent infinite sequence of polynomials in $\G_{00}$.

  The sequence $\{ q_0, q_1, \dots \}$ is an infinite sequence in
  $\G_0$ with $\{ q_0^e, q_1^e, \dots \}$  a linearly independent sequence of even functions.
  Each $q_k$ is continuously differentiable
  on $[-1,0]$ and on $[0,1]$ and satisfies $q_k(0) = 0$. \end{lem}

  \begin{proof} The results for the odd power polynomials $\{ p_0, p_1, \dots \}$ are obvious.

  If $h_1, h_2 \in \G_{0}$ with $h_1^* \not= h_2$, then $h_1 \odot h_2 \in \G_0 \setminus \G_{00}$.
It is in $\G_0$ by  (\ref{inteq2}).
For $t \in [0,1]$,
\begin{align}\label{inteq15}
\begin{split}
 (h_1 \odot h_2)^e(t) \ = \ &\frac{1}{2}[(h_1 \odot h_2)(t) +  (h_1 \odot h_2)(-t)] \ = \\
 \frac{1}{4}[(h_2(2t - & 1) + 1) + (h_1(-2t + 1) - 1)] \\ = \ \frac{1}{4}&[ h_2(2t  - 1) - h^*_1(2t - 1)].
\end{split}
\end{align}
In particular, this applies when $h_1$ and $h_2$ are distinct elements of $\G_{00}$.

From (\ref{inteq15}) it is clear that the sequence $\{ q_0^e, q_1^e, \dots \}$  is linearly independent. Differentiability on
$[-1,0]$ and $[0,1]$ is obvious.

 \end{proof} \vspace{.5cm}

For an $n$-tuple $\mathbf{f} = (f_1, \dots, f_n) \in \G^n$,  we define the continuous
linear map $L_{\mathbf{f}} : \CC([-1,1]) \to \R^n$ \index{$L_{\mathbf{f}}$}by
\begin{equation}\label{inteq16b}
L_{\mathbf{f}}(g) \ = \ (\int_{-1}^1 g(f_1^{-1}(t))) \ dt, \dots, \int_{-1}^1 g(f_n^{-1}(t))) \ dt).
\end{equation}
Thus, for $g \in \G$, $L_{\mathbf{f}}(g)  = (Q(g,f_1), \dots, Q(g,f_n))$.

\begin{df}\label{newdef7a} We define the following sets of $n$-tuples in $\G_0^n$.
  \begin{align}\label{inteq11aba}
  \begin{split}
  \mathcal{L}\mathcal{I}\mathcal{N}_n \ = \ \{ \mathbf{f} = \ &(f_1, \dots, f_n)  \in \G_0^n : \\
  \{ i, f_1, \dots, f_n \}  \ & \text{is linearly independent} \}, \\
  \mathcal{L}\mathcal{I}\mathcal{N}_n^+ \ = \ \{ \mathbf{f} = \ &(f_1, \dots, f_n)  \in \G_0^n : \\
  \{ f_1^e, \dots, f_n^e \}  \ & \text{is linearly independent} \},
  \end{split}\end{align}
  \index{$\mathcal{L}\mathcal{I}\mathcal{N}_n$}  \index{$\mathcal{L}\mathcal{I}\mathcal{N}_n^+$}
  \end{df}

 \begin{lem}\label{newlem4} If for $\{f_1, \dots, f_n \} \subset \G_0$ the
 sequence $\{f_1^e, \dots, f_n^e \}$ is linearly independent, then
 $\{i, f_1, \dots, f_n \}$ is linearly independent, i.e.
 $\mathcal{L}\mathcal{I}\mathcal{N}_n^+ \subset \mathcal{L}\mathcal{I}\mathcal{N}_n$. \end{lem}

  \begin{proof} If $z_0 i + \sum_{k=1}^n z_k f_k = 0$ then $\sum_{k=1}^n z_k f_k^e = 0$
  and so by linear independence of the even list,
  $z_k = 0$ for $k = 1, \dots, n$.  Since $i \not= 0$, it follows that $z_0 = 0$ as well.

 \end{proof} \vspace{.5cm}

 \begin{lem}\label{newlem5} Assume that $\mathbf{f} = (f_1, \dots, f_n) \in  \mathcal{L}\mathcal{I}\mathcal{N}_n$.
 There exists $g \in \G_{00}$ such that
 $\{i, f_1, \dots, f_n, g \} \in  \mathcal{L}\mathcal{I}\mathcal{N}_{n+1}$ and $L_{\mathbf{f}}(g) = 0$, i.e.
 $Q(g,f_k) = 0$ for $k = 1,\dots, n$.

 If $\mathbf{f} \in \mathcal{L}\mathcal{I}\mathcal{N}_{n}^+$, then
 there exists $g \in \G_{0}$ such that
 $\{ f_1^e, \dots, f_n^e, g^e \} \in \mathcal{L}\mathcal{I}\mathcal{N}_{n+1}^+$ and $L_{\mathbf{f}}(g) = 0$.
\end{lem}

  \begin{proof} We use the sequences from Lemma \ref{newlem3}.

  Let $V$ be the linear subspace of $\CC([-1,1])$ spanned by $p_{k+1} - p_{0}$ for $k = 0, \dots, 2n+1$.
  Thus, $V$ is a vector space
  of odd polynomials with the dimension of $V$ equal to $2n + 2$. Furthermore, $p(\pm 1) = 0$ for all $p \in V$.

 The restriction of the linear map $L_{\mathbf{f}}$ defines a linear map from $V $ to $\R^{n}$
  has a kernel with dimension at least $n +2$.

    On the other hand, the intersection of $V$ with the subspace spanned by
    $\{i, f_1, \dots, f_n \}$ has dimension at most $n + 1$.

    It follows that there exists $p$ in the kernel of $L_{\mathbf{f}}$ such
    that $\{i, f_1, \dots, f_n,$ $ p \}$ is linearly independent. Multiplying by
    a suitably small positive constant we may assume that the absolute value
    of the derivative of $p$ is bounded by $\frac{1}{2}$ on $[-1,1]$.

    Let $g = i + p$. Since $p(\pm 1) = 0$, $g(\pm 1) = \pm 1$ and since the
    derivative of $g$ is at least $\frac{1}{2}$ on $[-1,1]$, $g$
    is increasing.  Thus, $g \in \G$.  It is odd and so is in $\G_{00}$.
    Since $i \in \{i, f_1, \dots, f_n \}$, it follows that
    $\{i, f_1, \dots, f_n, g \}$ is linearly independent. Since $Q(i,f) = 0 $
    for all $f \in \G_0$, it follows that $Q(g,f_k) = 0$ for
    $k = 1, \dots, n$.

    Now assume that  $\{f_1^e, \dots, f_n^e \}$ is  linearly independent.
    This time let $V$ be the vector space spanned by
    $q_k^e$ for $k = 0, \dots, 2n$ and so with dimension $2n + 1$. Each
    $p \in V$ is an even function with $\int_{-1}^1 p(t) \ dt = 0$ and $p(\pm 1) = 0$.
    In addition, each function is continuously differentiable on $[-1,0]$ and on $[0,1]$.
    The kernel of the restriction of $L_{\mathbf{f}}$ has dimension
    at least $n+1$.

    The intersection of $V$ with the subspace spanned by  $\{f_1^e, \dots, f_n^e \}$ has dimension at most $n$.

     It follows that there exists $p$ in the kernel of $L_{\mathbf{f}}$ such that
     $\{f_1^e, \dots, f_n^e, p \}$ is linearly independent. Multiplying by
    a suitably small positive constant we may assume that the absolute value of
    the derivative of $p$ is bounded by $\frac{1}{2}$ on $[-1,0]$ and on
    $[0,1]$.  As before, let $g = i + p \in \G$.  Since $\int_{-1}^1 p(t) = 0$,
    $g \in \G_0$. Furthermore, $g^e = p$ and so
    $\{f_1^e, \dots, f_n^e, g^e \}$ is linearly independent.  As before, $Q(g,f_k) = 0$ for
    $k = 1, \dots, n$.

 \end{proof} \vspace{.5cm}

We will need a bit of linear algebra folklore.

\begin{lem}\label{genericlem0a} Let $V$ be a normed linear space.

(a) If $V_0$ is a finite dimensional subspace of $V$, then $V_0$ is closed in $V$.

(b) For any positive integer $n$, the set $\{ (v_1, \dots, v_n) \in V^n : \{ v_1, \dots, v_n \} $
is linearly independent $ \}$
is a $G_{\delta}$ subset of $V^n$.
\end{lem}

\begin{proof} (a) Let $v_1, \dots, v_n$ be a basis for $V_0$ and let
$M = \max_i || v_i ||$. For $(x_1, \dots, x_n) \in \R^n$
let $J(x_1, \dots, x_n) = \sum_i x_i v_i$. The linear isomorphism $J$ is
continuous, since the linear operations are continuous
in $V$.
Define $|| (x_1, \dots, x_n) ||_1 = || J(x_1, \dots, x_n) || \leq nM \cdot || (x_1, \dots, x_n) ||_0$
where $|| \cdot ||_0$ is the
Euclidean norm on $\R^n$.   Since $|| \cdot ||_1$ is continuous, and the unit sphere
in $\R^n$ is compact, there exists $m > 0$ such that
$|| \cdot ||_1 \geq m \cdot || \cdot ||_0$ on $\R^n$.  It follows that with respect
to the norm $|| \cdot ||_1$, $\R^n$ is complete and
so $V_0$ is a complete subspace of $V$. A  subset of a metric space is  closed if it is complete.

(b) For $K, k $  positive integers, let $A^k_K = \{ (x_1,\dots, x_k) \in \R^k : K^{-1} \leq |x_i| \leq K$
for $i \in [k] \}$. The set
$(v_1, \dots, v_k) \in V^k$ such that $\sum_i x_i v_i = 0$ for some $(x_1,\dots, x_k) \in A^k_K$
is closed because $A^k_K$
is compact. Take the union over the positive integers $K$, we obtain an $F_{\s}$ subset $W_k$ of
$V^k$. An $n$-tuple $(v_1, \dots, v_n) \in V^n$
is linearly dependent if and only if it projects to $W_k$ for some $k \leq n$, where the
projection omits $n-k$ vectors and renumbers the remaining $k$.
It follows that the set of linearly dependent $n$-tuples is an $F_{\s}$.
\end{proof}
\vspace{.5cm}

Finally, we recall a version of the Separating Hyperplane Theorem.

\begin{theo}\label{newtheo6} Let $C$ be a convex subset of $\R^n$ with $0 \in C$.  If
$0 \not\in C^{\circ}$, then there
exists a nonzero vector $z \in \R^n$ such that $\sum_{k=1}^n \ z_kc_k \leq 0$ for all $c \in C$. \end{theo}

\begin{proof}  Let $V$ be the affine subspace generated by $C$.  Since $0 \in C$, $V$ is a linear subspace of $\R^n$.

{\bfseries Case 1}($V$ is a proper subspace of $\R^n$): If $z$ a nonzero vector perpendicular to $V$, then
$\sum_{k=1}^n \ z_kc_k = 0$ for all $c \in C$.

{\bfseries Case 2}($V =\R^n$): For a convex set $C$, the \emph{relative interior} $ri C$
is the set of points $c \in C$ such that
there exists an open set $U$ in $\R^n$ with $c \in U$ such that $U \cap V \subset C$.
When $V = \R^n$ this is just the interior
$C^{\circ}$.  So if $0 \not\in C^{\circ}$, then $0 \not\in ri C$ and the Theorem of the
Separating Hyperplane,
\index{separating hyperplane} (see, e.g. \cite{MT} p. 38) says that there exist $z \in \R^n \setminus \{ 0 \}$ and
$b \in \R$ such that $\sum_{k=1}^n \ z_k 0 \geq b$ and $\sum_{k=1}^n \ z_kc_k \leq b$ for all $c \in C$. Since $0 \in C$
it is clear that $b = 0$.

\end{proof}

 \vspace{1cm}

 \section{Tournaments of Generic $n$-tuples}
\vspace{.5cm}

The following example illustrates why we will focus upon the proper elements of $\G$, i.e. the elements of $\G_0$.

\begin{ex}\label{ex} Not every edge of the digraph $\Gamma_{\G}$ is contained in a $3$-cycle. \end{ex}

\begin{proof} Recall that $i \in \G_{00}$ is the identity with $i(t) = t$. Assume that $g \in \G$ with
$g \geq i$ and $g \not= i$.  That is, $g(t) \geq t$ for all $t \in [-1,1]$ and the inequality is strict for
some $t$. It follows that $\int_{-1}^1 g(t) \ dt > 0$, i.e. $g \circ i^{-1} \in \G_+$. Thus, $i \to g$ in
$\Gamma_{\G}$. If $g \to f$, then
\begin{equation}\label{inteq7}
f \circ i^{-1} = f = (f \circ g^{-1}) \circ g \geq f \circ g^{-1},
\end{equation}
and the latter inequality is not an equation.  It follows that
\begin{equation}\label{inteq8}
\int_{-1}^1 f(t) \ dt \ > \ \int_{-1}^1 f(g^{-1}(t)) \ dt \ > \ 0.
\end{equation}
Hence, $i \to f$ and  $f \not\in \G_0$.

Thus, $(i,g)$ is not contained in a $3$-cycle.  Furthermore, there does not exists $f \in \G_0$ such that $g \to f$.

\end{proof}
\vspace{.5cm}

 In contrast with the example is the following result which is essentially the continuous
 time version of Theorem 2 of \cite{FT}.
 It

 \begin{theo}\label{fttheo1}
If $f \in \G_0 \setminus \{ i \}$, then there exist $g_1, g_2 \in \G_{0}$ such that
\begin{equation}\label{inteq9}
  \int_{-1}^1 f(g_1^{-1}(t)) \ dt \ > \ 0 \ > \  \int_{-1}^1 f(g_2^{-1}(t)) \ dt.
\end{equation}
That is,  $Q(f,g_1) > 0 > Q(f,g_2)$ and so  $g_1 \to f \to g_2$ in $\Gamma_{\G}$.

If $f \in \G_0 \setminus \G_{00}$, then $g_1$ and $g_2$ can be chosen in $\G_{00}$.
\end{theo}

\begin{proof} This follows easily from Lemma \ref{newlem1} applied to $f - i$ and to $f^e$,
but we will omit the details and instead
derive the result later from a more general theorem.

\end{proof}

  {\bfseries Remark:} Note that $Q(g,i) = 0$ for all $g \in \G_0$ and so there does
  not exist $g \in \G_0$ such that $g \to i$ or $i \to g$ in
$\Gamma_{\G}$.
\vspace{.5cm}

For an $n$-tuple $\mathbf{f} = (f_1, \dots, f_n) \in \G^n$, we define the
\emph{associated digraph}  $R[\mathbf{f}]$  \index{$R[\mathbf{f}]$}
on $[n]$ by
\begin{equation}\label{inteq16}
(i,j) \ \in R[\mathbf{f}] \ \Longleftrightarrow \  Q(f_j,f_i) \ > \ 0
  \ \Longleftrightarrow \ (f_i, f_j) \in \Gamma_{\G}.
\end{equation}
\index{digraph associated to $\mathbf{f}$}\index{digraph!associated}.
That is, $R[\mathbf{f}]$ is the digraph obtained by
pulling back  $\Gamma_{\G}$ from $\G$ to $[n]$ via the map $i \mapsto f_i$.

For $\mathbf{f}$ we define the \emph{associated  matrix} $Q[\mathbf{f}] $ \index{$Q[\mathbf{f}]$}  by
\begin{equation}\label{inteq16a}
Q[\mathbf{f}]_{ij} \ = \ Q(f_j,f_i).
\end{equation}
\index{matrix associated to $\mathbf{f}$}. Thus, $Q[\mathbf{f}] $ is a real $n \times n$ anti-symmetric matrix with
$Q[\mathbf{f}]_{ij} > 0$ if and only if $f_i \to f_j$ in $\Gamma_{\G}$.


We extend Definition \ref{newdef7a}.

\begin{df}\label{newdef7} We define the following sets of $n$-tuples in $\G^n$.
 \begin{align}\label{inteq11aa}
  \begin{split}
  \mathcal{T}\mathcal{O}\mathcal{U}\mathcal{R}_n \ = \ \{ \mathbf{f} = \ &(f_1, \dots, f_n) \in \G^n : \\
  Q(f_i,f_j) \not= 0 \ &\text{for all} \ i,j \in [n] \ \text{with} \ i \not= j \}
  \end{split}\end{align}
  \index{$\mathcal{T}\mathcal{O}\mathcal{U}\mathcal{R}_n$}
  \begin{align}\label{inteq11abab}
  \begin{split}
  \mathcal{L}\mathcal{I}\mathcal{N}_n \ = \ \{ \mathbf{f} = \ &(f_1, \dots, f_n)  \in \G_0^n : \\
  \{ i, f_1, \dots, f_n \}  \ & \text{is linearly independent} \}, \\
  \mathcal{L}\mathcal{I}\mathcal{N}_n^+ \ = \ \{ \mathbf{f} = \ &(f_1, \dots, f_n)  \in \G_0^n : \\
  \{ f_1^e, \dots, f_n^e \}  \ & \text{is linearly independent} \},
  \end{split}\end{align}
  \index{$\mathcal{L}\mathcal{I}\mathcal{N}_n$}  \index{$\mathcal{L}\mathcal{I}\mathcal{N}_n^+$}
  \end{df}

Thus, $\mathbf{f} \in \mathcal{T}\mathcal{O}\mathcal{U}\mathcal{R}_n$ if and
only if the digraph $R[\mathbf{f}]$ is a tournament on $[n]$
in which case it is called
the
\emph{tournament associated to }$\mathbf{f} \in \mathcal{T}\mathcal{O}\mathcal{U}\mathcal{R}_n$.
\index{tournament associated to $\mathbf{f}$}\index{tournament!associated}.

From Lemma \ref{newlem4} we see that
$ \mathcal{L}\mathcal{I}\mathcal{N}_n^+ \subset  \mathcal{L}\mathcal{I}\mathcal{N}_n$. Notice that for these
we restrict to $n$-tuples in $\G_0^n$.

Our major tool is the following theorem.

\begin{theo}\label{newtheo8} If $\mathbf{f} = (f_1, \dots, f_n)  \in  \mathcal{L}\mathcal{I}\mathcal{N}_n$,
then the interior
$L_{\mathbf{f}}(\G_0)^{\circ}$ is a convex subset of $\R^n$ containing $0 \in \R^n$. Furthermore, this set
equals $L_{\mathbf{f}}(\{ g \in \G_0: (f_1, \dots, f_n, g) \in  \mathcal{L}\mathcal{I}\mathcal{N}_{n+1} \})^{\circ}$.

If $\mathbf{f} =  (f_1, \dots, f_n)  \in  \mathcal{L}\mathcal{I}\mathcal{N}_n^+$, then
$L_{\mathbf{f}}(\G_{00})^{\circ}$ is a convex subset of $\R^n$ containing $0 \in \R^n$. Furthermore,
$L_{\mathbf{f}}(\G_0)^{\circ} =  L_{\mathbf{f}}(\{ g \in \G_0:
 (f_1, \dots, f_n, g) \in  \mathcal{L}\mathcal{I}\mathcal{N}_{n+1}^+ \})^{\circ}$.
\end{theo}

\begin{proof} {\bfseries Step 1:} $L_{\mathbf{f}}(\G_0)$ is a convex set containing $0$ in its interior and if
$\mathbf{f}  \in  \mathcal{L}\mathcal{I}\mathcal{N}_n^+$, then $L_{\mathbf{f}}(\G_{00})$ is a convex set containing
$0$ in its interior. \vspace{.25cm}

Since $L_{\mathbf{f}}$ is linear and $\G_{0}, \G_{00}$ are convex, their images are convex. Since $i \in \G_{00}$ and
$L_{\mathbf{f}}(i) = 0$, each convex image contains $0$.

If $0$ is not in $L_{\mathbf{f}}(\G_{0})^{\circ}$, then by the Separating Hyperplane Theorem, Theorem \ref{newtheo6},
there exists $z \in \R^n \setminus \{ 0 \}$ such that $\sum_{k=1}^n z_k Q(g, f_k) \leq 0$ for all $g \in \G_0$.
Hence, by (\ref{coneq4}) $\sum_{k=1}^n z_k Q( f_k,g) \geq 0$ for all $g \in \G_0$.  So with $F = \sum_{k=1}^n z_k f_k$,
$\int_{-1}^1 F(g^{-1}(t)) \ dt \geq 0$ for all $g \in \G_0$.  Let $z_0 = - \sum_{k=1}^n z_k$ and $\bar F = F + z_0 i$.
Since $\{i, f_1, \dots, f_n \} \subset \G_0$, $\int_{-1}^1 \bar F(t) \ dt = 0$, $\bar F(\pm 1) = 0$ and
$\int_{-1}^1 \bar F(g^{-1}(t)) \ dt \geq 0$ for all $g \in \G_0$. From Lemma \ref{newlem1} it follows that $\bar F = 0$.
That is, $z_0 i + \sum_{k=1}^n z_k f_k = 0$. This contradicts the
assumption $\mathbf{f}  \in  \mathcal{L}\mathcal{I}\mathcal{N}_n$.

If $0$ is not in $L_{\mathbf{f}}(\G_{00})^{\circ}$, then, as before,
there exists $z \in \R^n \setminus \{ 0 \}$ such that $\sum_{k=1}^n z_k Q(f_k, g) \geq 0$ for all $g \in \G_{00}$.
Since $g^* = g$, we have
 \begin{equation} \label{neweq8}
 Q(f_k,g) \ = \ - Q(f_k^*,g^*) \ =  - Q(f_k^*,g) \ = \ \int_{-1}^1 f_k^e(g^{-1}(t)) \ dt.
 \end{equation}
So with $F = \sum_{k=1}^n z_k f_k^e$,
$\int_{-1}^1 F(g^{-1}(t)) \ dt \geq 0$ for all $g \in \G_{00}$. Because $f_k^e(\pm 1) = 0$ and
$\int_{-1}^1 f_k^e(t) \ dt = \int_{-1}^1 f_k(t) \ dt = 0$, we have
$\int_{-1}^1  F(t) \ dt = 0$, $ F(\pm 1) = 0$ with $F$ even and
$\int_{-1}^1 F(g^{-1}(t)) \ dt \geq 0$ for all $g \in \G_{00}$. Again
Lemma \ref{newlem1} implies $\sum_{k=1}^n z_k f_k^e = F = 0$.
So if  $\mathbf{f}  \in  \mathcal{L}\mathcal{I}\mathcal{N}_n^+$, it
follows that  $0 \in L_{\mathbf{f}}(\G_{00})^{\circ}$.
\vspace{.25cm}

{\bfseries Step 2:} $L_{\mathbf{f}}(\G_0)^{\circ} \ = \ $
$L_{\mathbf{f}}(\{ g \in \G_0: (f_1, \dots, f_n, g) \in  \mathcal{L}\mathcal{I}\mathcal{N}_{n+1} \})^{\circ}$.
If $\mathbf{f}  \in  \mathcal{L}\mathcal{I}\mathcal{N}_n^+$, then this set equals
$L_{\mathbf{f}}(\{ g \in \G_0: (f_1, \dots, f_n, g) \in  \mathcal{L}\mathcal{I}\mathcal{N}_{n+1}^+ \})^{\circ}$.
\vspace{.25cm}

Observe first that if $C \subset \R^n$ is a convex set with $0 \in C^{\circ}$,
then $C^{\circ}$ is convex (see Proposition 1 of \cite{MT}) and
clearly,
\begin{equation} \label{neweq9}
C^{\circ} \ = \ \bigcup_{0 < x < 1} x \cdot C^{\circ}.
\end{equation}

We can apply Lemma \ref{newlem5} to get $g_0 \in \G_{00}$ such that
$L_{\mathbf{f}}(g_0) = 0$ and $(f_1, \dots,f_n, g_0) \in \mathcal{L}\mathcal{I}\mathcal{N}_{n+1}$.
For any $g \in \G_0$, if $(f_1, \dots,f_n, g) \not\in \mathcal{L}\mathcal{I}\mathcal{N}_{n+1}$, then
$g$ is a linear combination of $i, f_1, \dots, f_n$ and so for any $x \in [0,1)$
$(f_1, \dots,f_n, g_x) \in \mathcal{L}\mathcal{I}\mathcal{N}_{n+1}$ where $g_x = x g + (1 - x)g_0$.
Since $L_{\mathbf{f}}(g_x) = x \cdot L_{\mathbf{f}}(g)$, it follows that
$L_{\mathbf{f}}(\{ g \in \G_0: (f_1, \dots, f_n, g) \in  \mathcal{L}\mathcal{I}\mathcal{N}_{n+1} \})^{\circ}$
contains $\bigcup_{0 < x < 1} x \cdot L_{\mathbf{f}}(\G_0)^{\circ} = L_{\mathbf{f}}(\G_0)^{\circ}$.
The reverse inclusion is obvious.

If $\mathbf{f}  \in  \mathcal{L}\mathcal{I}\mathcal{N}_n^+$, then we can apply
Lemma \ref{newlem5} to get $\bar g_0 \in \G_{0}$ such that
$L_{\mathbf{f}}(\bar g_0) = 0$ and $(f_1, \dots,f_n, \bar g_0) \in \mathcal{L}\mathcal{I}\mathcal{N}_{n+1}^+$.
As before, we define $\bar g_x =  x g + (1 - x)\bar g_0$ and if
$(f_1, \dots,f_n, g) \not\in \mathcal{L}\mathcal{I}\mathcal{N}_{n+1}^+$, then
$(f_1, \dots,f_n, \bar g_x) \in \mathcal{L}\mathcal{I}\mathcal{N}_{n+1}$ for all $x \in [0,1)$.
The result follows as before.

\end{proof} \vspace{.5cm}


\begin{cor}\label{fttheo3} ({\bfseries Lifting Theorem}) Let
 $ \mathbf{f} = (f_1, \dots, f_n)  \in  \G_0^n $ and let $J \subset [n]$.

If $\mathbf{f}  \in  \mathcal{L}\mathcal{I}\mathcal{N}_n$, then  there exists
$g \in \G_{0}$ such that
\begin{equation}\label{inteq12}
 Q(g,f_j)  \ \begin{cases} < 0 \quad \text{for} \ j \in J,
 \\ > 0 \quad \text{for} \ j \in [n] \setminus J. \end{cases}
 \end{equation}

 In terms of the digraph $\Gamma_{\G}$,  $g \to f_j$ for $j \in J$ and
 $f_j \to g$ for $j \in [n] \setminus J$.

 Furthermore, we can choose $g$ so that $(f_1, \dots, f_n, g)  \in \mathcal{L}\mathcal{I}\mathcal{N}_{n+1}$.

 If
$\mathbf{f}  \in  \mathcal{L}\mathcal{I}\mathcal{N}_n^+$, then $g$ can be chosen in $\G_{00}$. Alternatively, it
can be chosen so that $(f_1, \dots, f_n, g)  \in \mathcal{L}\mathcal{I}\mathcal{N}_{n+1}^+$.

 \end{cor}

 \begin{proof} Define $z \in \R^n$ by
 $$z_j = \begin{cases} -1 \ \text{for} \ j \in J, \\ +1 \ \text{for} \ j \in [n] \setminus J. \end{cases}$$

 By Theorem \ref{newtheo8} some positive multiple of $z$ lies in
 $L_{\mathbf{f}}(\{ g \in \G_0: (f_1, \dots, f_n, g) \in  \mathcal{L}\mathcal{I}\mathcal{N}_{n+1} \})^{\circ}$.
 That is, there exists $ g \in \G_0$ with $ (f_1, \dots, f_n, g) \in  \mathcal{L}\mathcal{I}\mathcal{N}_{n+1}$
 and with $L_{\mathbf{f}}(g)$ a
 positive multiple of $z$.

 If $\mathbf{f} \in \mathcal{L}\mathcal{I}\mathcal{N}_{n}^+$, then by the theorem we can choose $g \in \G_{00}$ or
 choose so that  $ (f_1, \dots, f_n, g) \in \mathcal{L}\mathcal{I}\mathcal{N}_{n+1}$ and with either
 choice obtain $L_{\mathbf{f}}(g)$ as a
 positive multiple of $z$.

 Notice that the latter two possibilities are mutually exclusive since $g \in \G_{00}$ implies $g^e = 0$ and so implies
 $ \{ f_1^e, \dots, f_n^e, g^e \}$ is linearly dependent.

\end{proof}
\vspace{.5cm}

In contrast with Example \ref{ex} we have the following.

\begin{cor}\label{ftcor2a} (a) If $f \in \G_0 \setminus \{ i \}$, then there exist $g_1, g_2 \in \G_{0}$ such that
\  $Q(g_1,f) < 0 < Q(g_2,f)$ and so  $g_1 \to f \to g_2$ in $\Gamma_{\G}$.

If $f \in \G_0 \setminus \G_{00}$, then $g_1$ and $g_2$ can be chosen in $\G_{00}$.

(b) If $f_1, f_2 \in \G_0$ with $f_1 \to f_2$ in $\Gamma_{\G}$, then there exists $f_3 \in \G_0$ such that
$f_2 \to f_3 \to f_1$ in $\Gamma_{\G}$.  That is, every edge of the restriction  $\Gamma_{\G}|\G_0$ is contained in a
$3$-cycle in $\Gamma_{\G}|\G_0$ . \end{cor}

\begin{proof} (a): If $f \in \G_0 \setminus \{ i \}$, then $\{ i, f \}$ is linearly independent and so from
Corollary \ref{fttheo3} it follows that there exist $g_1, g_2 \in \G_0$ such that $g_1 \to f \to g_2$.

If $f \in \G_0 \setminus \G_{00}$, then $f^e \not= 0$ and so $\{ f^e \}$ is linearly independent.
By Corollary \ref{fttheo3} again
we can choose $g_1, g_2 \in \G_{00}$.

(b): If $f_1 = x_1 i + x_2 f_2$, then $Q(f_1, f_2) = x_1 Q(i, f_2) + x_2 Q(f_2, f_2) = 0$.  So
$f_1 \to f_2$ implies that $\{ i, f_1, f_2 \}$ is linearly independent and so the existence of the
required $f_3$ follows from
Corollary \ref{fttheo3}.

\end{proof}

{\bfseries Remark:} Observe that part (a) is a restatement of Theorem \ref{fttheo1}. \vspace{.5cm}

\begin{ex}\label{ftex3}For $n \geq 3$ there exist linearly dependent $n$-tuples
in $\G_0^n \cap \mathcal{T}\mathcal{O}\mathcal{U}\mathcal{R}_n$.
Nonetheless, the
general lifting result of Corollary \ref{fttheo3} requires that  $\{ i, f_1, \dots, f_n  \}$
be linearly independent. \end{ex}

\begin{proof} Assume $f_1 \to f_0$ with $f_0, f_1 \in \G_0$.

For $x \in [0,1]$ let $f_x = (1 - x)f_0 + x f_1$. Because
$-Q(f_1,f_0) = Q(f_0,f_1) > 0$ we have, for $a < b$,

$$Q(f_a,f_b) = b(1-a)Q(f_0,f_1) + a(1-b)Q(f_1,f_0) = (b - a)Q(f_0,f_1) > 0.$$

That is, $f_b \to f_a$ when $b > a$.
If $0 < a_1 < \dots < a_{n-2} < 1$, then
$(f_0, f_{a_1}, \dots, f_{a_{n-2}}, f_1) \in  \G_0^n \cap \mathcal{T}\mathcal{O}\mathcal{U}\mathcal{R}_n$.

Suppose that $\mathbf{f} = (f_1, \dots, f_n)\in \G_0^n \cap \mathcal{T}\mathcal{O}\mathcal{U}\mathcal{R}_n$.
Assume that for $J_1 \subset [n-1]$ and $J_2 = [n-1] \setminus J_1$, $f_n = \sum_{j \in J_1} a_j f_j - \sum_{j \in J_2} a_j f_j + a_0 i$
with $a_j \geq 0$ for $j = 1, \dots n$. Since $Q(f,i) = 0$ for all $f \in \G_0$, $f_n \not= i$ and so  $a_j > 0$ for at least one
 $j \in [n]$. Now if $g \in \G_0$ with $Q(f_j,g) > 0$ for all $j \in J_1$ and $Q(f_j,g) < 0$ for all $j \in J_2$, then
 $Q(f_n,g) =  \sum_{j \in J_1} a_j Q(f_j,g) - \sum_{j \in J_2} a_j Q(f_j,g) > 0$.  That is, $g \to f_j$ for all $j \in J_1$ and
 $f_j \to g$ for all $j \in J_2$ implies $g \to f_n$. So we cannot lift with $J = J_1 \subset [n]$  and $g \in \G_0$.

 \end{proof}
\vspace{.5cm}

Recall from Proposition \ref{elemgrpprop} (ii) that $\G$ is a dense, $G_{\d}$ subset of
$\bar \G$ which is closed in the
complete, separable metric space $\CC([-1,1])$. It follows that the relatively
closed subsets $\G_0 = \G \cap \bar \G_0$ and
$\G_{00} = \G \cap \bar \G_{00}$ are $G_{\d}$ subsets of $\CC([-1,1])$ as well.

\begin{cor}\label{ftcor2} (a) For every positive integer $n$ the set of $n$-tuples
 $  \mathcal{T}\mathcal{O}\mathcal{U}\mathcal{R}_n$
 is a relatively open,
  dense subset of  $\G^n$ and the intersection with $\G_0^n$ is dense in $\G_0^n$.

  (b)  For every positive integer $n$ the sets of $n$-tuples
   $\mathcal{L}\mathcal{I}\mathcal{N}_n^+ \subset \mathcal{L}\mathcal{I}\mathcal{N}_n$
are dense $G_{\d}$ subsets of $\G_0^n$.
\end{cor}

\begin{proof} (a) Since the intersection of finitely many open, dense subsets
is open and dense, it suffices to prove this for the case $n = 2$.

Because the group operations are continuous and  $Q(f_1,f_2) = $ \\ $\int_{-1}^1 f_1(f^{-1}_2(t)) \ dt$,
the condition $Q(f_1,f_2) \ \not= \ 0$ is an open condition on the pair
$(f_1, f_2) \in \G^2 $.

Density in $\G^2$ is easy.  If $Q(f_1,f_2) = 0$, choose $h \in \G$ with $h \geq f_1$,
but $h \not= f_1$. By replacing $h$ by
$xh + (1-x)f_1$ with small $x > 0$ we may choose $h$ arbitrarily close to $f_1$.
Since $h \circ f_2^{-1} \geq f_1 \circ f_2^{-1}$ and
is distinct from it, $Q(h,f_2) > 0$.

For a pair $(f_1,f_2) \in \G_0^2$ we may  first perturb to get $f_1 \in \G_0 \setminus \{ i \}$.
Now if $Q(f_1,f_2) = 0$,
then by Theorem \ref{fttheo1} there exists $g \in \G_{0}$ such that $Q(f_1,g) \not= 0$.
Let $f_2^x = x g + (1-x) f_2$ so that by (\ref{coneq6})
$Q(f_1,f_2^x) \ = \ x Q(f_1,g)$ and this is not equal to zero for
$x > 0$.

(b) The inclusion is Lemma \ref{newlem4}.

By continuity of the linear map $f \to f^e$, it follows from Lemma \ref{genericlem0a}(b)
 that the sets are $G_{\d}$.

From the inclusion it suffices to show that $\mathcal{L}\mathcal{I}\mathcal{N}_n^+$ is dense.

By induction on $k$ we
show we can perturb to get linear independence when $\{ f_1^e, \dots, f_{n-k}^e \}$
is linearly independent. No perturbation is
needed when $k = 0$. Now assuming the result for $k-1$, if $\{ f_1^e, \dots, f_{n-k}^e, f_{n-(k-1)}^e \}$
is linearly independent
then we can apply the induction hypothesis to perturb to linear independence.  If it is not, then
$f_{n-(k-1)}^e$ is a linear combination of the linearly independent set $\{ f_1^e, \dots, f_{n-k}^e \}$.
For a sequence $\{q_0, q_1, \dots \}$ as constructed in
Lemma \ref{newlem3}, at least one among $q_0^e, \dots, q_{n-(k-1)}^e $ is not
a linear combination of $\{ f_1^e, \dots, f_{n-k}^e \}$. Suppose $q_i^e$ is not. Replace $f_{n-(k-1)}$ by
$ f^x_{n-(k-1)} = (1-x)f_{n-(k-1)} + x q_i$. For any $x > 0$  $\{ f_1^e, \dots, f_{n-k}^e, (f^x_{n-(k-1)})^e \}$
is linearly independent.
Choose $x > 0$ small and then apply the induction hypothesis as before.

\end{proof}
\vspace{.5cm}

\begin{df}\label{genericdef2} Call $\mathbf{f} = (f_1, \dots, f_n)$ a \emph{generic}
$n$-tuple \index{generic $n$-tuple}when
$\mathbf{f} \in \mathcal{L}\mathcal{I}\mathcal{N}_n \cap \mathcal{T}\mathcal{O}\mathcal{U}\mathcal{R}_n$.
Call it \emph{strongly generic}\index{strongly generic $n$-tuple} if
$\mathbf{f} \in \mathcal{L}\mathcal{I}\mathcal{N}_n^+ \cap \mathcal{T}\mathcal{O}\mathcal{U}\mathcal{R}_n$.

We denote by $\G\mathcal{E}\mathcal{N}_n$ and $\G\mathcal{E}\mathcal{N}_n^+$ \index{$\G\mathcal{E}\mathcal{N}_n^+$ }
\index{$\G\mathcal{E}\mathcal{N}_n$ } the set of generic $n$-tuples \index{generic $n$-tuple} and the set of
strongly generic $n$-tuples.
\end{df}
\vspace{.5cm}

\begin{prop}\label{genericprop3} For every positive integer $n$ the sets of generic $n$-tuples
$\G\mathcal{E}\mathcal{N}_n^+ \subset  \G\mathcal{E}\mathcal{N}_n$
 are dense $G_{\d}$ subsets of $\G_0^n$.
\end{prop}

\begin{proof} This is immediate from  Corollary \ref{ftcor2} and the Baire Category Theorem.

\end{proof}
\vspace{.5cm}

\begin{theo}\label{generictheo4} Assume that $\mathbf{f} = (f_1, \dots, f_n)  \in \G\mathcal{E}\mathcal{N}_n $
with $R[\mathbf{f}]$
the associated tournament on $[n]$ and assume that $R$ is a tournament on $[n+1]$ whose
 restriction to $[n]$ is $R[\mathbf{f}]$.
For $f_{n+1} \in \G_0$ write $ \mathbf{f'}$ for the $n+1$-tuple $(f_1, \dots, f_n, f_{n+1})$.

(a) The set $\{ f_{n+1} \in \G_0 : \mathbf{f'}  \in \G\mathcal{E}\mathcal{N}_{n+1} \ \}$
 is  open and dense in $\G_0$. If $\mathbf{f} \in \G\mathcal{E}\mathcal{N}_n^+ $, then
 $\{ f_{n+1} \in \G_0 : \mathbf{f'}  \in \G\mathcal{E}\mathcal{N}_{n+1}^+ \ \}$  is  open and dense in $\G_0$.

 (b) The set $$\{ f_{n+1} \in \G_0 : \mathbf{f'}  \in \G\mathcal{E}\mathcal{N}_{n+1} \
 \text{, and } \ R[\mathbf{f'}] = R \ \}$$
 is  open  and nonempty in $\G_0$.

 (c)  If $\mathbf{f} \in \G\mathcal{E}\mathcal{N}_n^+ $, then the set
 $$\{ f_{n+1} \in \G_0 : \mathbf{f'}  \in \G\mathcal{E}\mathcal{N}_{n+1}^+ \
 \text{, and } \ R[\mathbf{f'}] = R \ \}$$ is open and nonempty
 in $\G_0$.

In particular,
every possible tournament extension of $R[\mathbf{f}]$ occurs as the associated
tournament of some  extension $\mathbf{f'}$ of $\mathbf{f}$
to a generic $n+1$-tuple.
\end{theo}

\begin{proof}  By Lemma \ref{genericlem0a}(a) the set of $f_{n+1}$ such that
$f_{n+1}$ lies in the space spanned by $\{i, f_1, \dots, f_n \}$
is closed in $\G_0$. Similarly, the set of $f_{n+1}$ such that $f_{n+1}^e$ lies
in the space spanned by $\{f_1^e, \dots, f_n^e \}$
is closed in $\G_0$. So the conditions on  $f_{n+1}$ that $\{i, f_1, \dots, f_n, f_{n+1} \}$ be linearly independent
(or that $\{f_1^e, \dots, f_n^e \}$ be linearly independent) is an open condition when
 $\mathbf{f}$ is generic (resp. when $\mathbf{f}$
is strongly generic).
The condition $Q(f_{n+1},f_i)  > 0$ or $< 0$ for any $i \in [n]$ is an open condition. It follows that
the $f_{n+1}$'s such that $\mathbf{f'}$ is generic or strongly generic form an
open set as do those with $R[\mathbf{f'}] = R$.

By Corollary \ref{fttheo3}, there exists $f_{n+1} \in \G_{0}$ such that
$f_{n+1} \to f_i$ for all $i \in R(n+1) \subset [n]$, i.e.
for those $i$ with $n+1 \to i$ in $R$,
and $f_i \to f_{n+1}$ for all $i$ in $R^{-1}(n+1)$,  the complementary subset in $[n]$ and
 such that $\mathbf{f'} \in \mathcal{L}\mathcal{I}\mathcal{N}_{n+1}$.
Since $\mathbf{f'}$ is clearly in $ \mathcal{T}\mathcal{O}\mathcal{U}\mathcal{R}_{n+1}$,
 it is in $\G\mathcal{E}\mathcal{N}_{n+1}$.

 If $\mathbf{f}$ is strongly generic then
$f_{n+1}$ can be chosen so that $\mathbf{f'}$ is in
$\mathcal{L}\mathcal{I}\mathcal{N}_{n+1}^+$ and so in $\G\mathcal{E}\mathcal{N}_{n+1}^+$.

%
%
%
%
%

For arbitrary $h \in \G_0$ the $n+1$-tuple
$(f_1, \dots, f_n,(1-x)h + xf_{n+1})$ is generic for $x > 0$ small enough (or strongly generic
if $\mathbf{f}$ is).
For example, if the original $n+1$-tuple with $x = 0$ has linearly
independent even parts, then it still does for small $x > 0$.
If it had linearly dependent even parts then they becomes linearly
 independent for $x \in (0,1]$. Similarly for the tournament inequalities.
 This completes the proof of
density in (a).

\end{proof}
\vspace{.5cm}

\begin{cor}\label{genericcor5} For every tournament $R$ on $[n]$ and
every $f_1 \in \G_0 \setminus \{ i \}$, there exists a  generic $n$-tuple $\mathbf{f}$
which begins with $f_1$ and which has associated tournament $R[\mathbf{f}]$ equal to $R$.
If $f_1 \in \G_0 \setminus \G_{00}$, we can choose $\mathbf{f}$
strongly generic.\end{cor}

\begin{proof} Use induction on $n$, beginning with $f_1$ for $n = 1$ and then
applying Theorem \ref{generictheo4}.

\end{proof}
%
\vspace{.5cm}

In particular, this completes the proof of  Theorem \ref{maintheo3} and so of Theorem \ref{maintheo2}.

By using Theorem \ref{newtheo8} we can sharpen these results.

\begin{df}\label{newdef9} We will call a subset $C$ of a vector space $V$
 \emph{star-shaped about $0$} \index{star-shaped about $0$}
if for every nonzero vector $v \in V$  the set $\{ x \in \R : x \cdot v \in C \}$
 is an interval in $\R$ containing $0$ in its interior.
\end{df}
\vspace{.5cm}

\begin{theo}\label{newtheo10} The sets of matrices
$\{ M[\mathbf{f}] : \mathbf{f} \in \mathcal{L}\mathcal{I}\mathcal{N}_{n}^+  \}$
$\subset \{ M[\mathbf{f}] : \mathbf{f} \in \mathcal{L}\mathcal{I}\mathcal{N}_{n}  \}$
are star-shaped about $0$ in the vector space
of real anti-symmetric $n \times n$ matrices. \end{theo}

 \begin{proof} Observe that for $\mathbf{f} = (f_1, \dots, f_n)$ we can define
  $\mathbf{f}^x = (xf_1 + (1-x)i, \dots, xf_n + (1-x)i)$
 and get $M[\mathbf{f}^x]  = x^2 \cdot M[\mathbf{f}]$. It thus suffices to show
  that for any nonzero anti-symmetric matrix $M$
 there exists $y > 0$ and $\mathbf{f} \in \mathcal{L}\mathcal{I}\mathcal{N}_{n}$
 such that $y M = M[\mathbf{f}]$.  Then there
exist $\bar y > 0$ and $\overline{\mathbf{f}}$ such that  $\bar y (-M) = M[\overline{\mathbf{f}}]$.
 Thus, $\{ M[\mathbf{f}] \}$ contains
 the interval $[-\bar y,y]$.

 Again we use induction on $n$.  The result is vacuously true for $n = 1$ since $0$
 is the only anti-symmetric $1 \times 1$ matrix.

 Assume the result is true for $n$ and let $M$ be an arbitrary $(n+1) \times (n+1)$ anti-symmetric matrix.
 Let $M'$ be the $n \times n$ principal
 minor so that $M'_{kj} = M_{kj}$ for $k,j \in [n]$. Let $L \in \R^n$ with $L_k = M_{k (n+1)}$ for $k \in [n]$.

 By induction hypothesis, there exists $\mathbf{f} \in \mathcal{L}\mathcal{I}\mathcal{N}_{n}$ and $y_1 > 0$ such that
 $ M[\mathbf{f}] = y_1 M'$.

 By Theorem \ref{newtheo8} there exists $f_{n+1} \in \G_0$ such that
 $\mathbf{f'} = (f_1, \dots, f_n,f_{n+1}) \in \mathcal{L}\mathcal{I}\mathcal{N}_{n+1}$
 and $y_2 > 0$ such that $L_{\mathbf{f}}(f_{n+1}) = y_2 L$.

 Choose $x \in (0,1)$ so that $z_1 = x y_2 < 1$ and $z_2 = xy_1 < 1$.
 Define $\bar f_k = z_1 f_k + (1 - z_1)i$ for $k \in [n]$
 and $\bar f_{n+1} = z_2 f_{n+1} + (1 - z_2)i$ and
 $\overline{\mathbf{f'}} = (\bar f_1, \dots, \bar f_n, \bar f_{n+1})$
 \begin{align} \label{neweq10}
 \begin{split}
 Q(\bar f_j,\bar f_k) \ = \ &z_1^2 Q(f_j,f_k) \ = \ z_1^2 y_1 M_{kj}, \\
  Q(\bar f_{n+1},\bar f_k) \ = \ &z_1z_2 Q(f_{n+1},f_k) \ = \ z_1z_2 y_2 M_{k(n+1)}.
  \end{split}
  \end{align}
 With $z = z_1^2 y_1 = z_1z_2 y_2$, we have $M[ \overline{\mathbf{f'}}] = z M$ completing the inductive step.

  \end{proof} \vspace{.5cm}

  Now let $\mathbf{f} = (f_1, \dots, f_n) \in \G_0^n$.  Recall that with
  $q : [-1,1] \to [0,1]$ defined by $q(t) = \frac{t + 1}{2}$
  we define $F_k = q \circ f_k \circ q^{-1}$ and let $X_k = F_k^{-1}(U_k)$ with $U_1, \dots, U_n$
  independent uniform random variables on $[0,1]$.
  Thus, $ \mathbf{X} = (X_1, \dots, X_n)$ is an $n$-tuple of independent, proper,
  continuous random variables on $[0,1]$ and
   \begin{equation} \label{neweq11}
   P(X_k > X_j) \ = \ \int_0^1 F_j(F_k^{-1}(x)) \ dx \ = \ \frac{1}{2} + \frac{1}{4} Q(f_j,f_k)
   \end{equation}

  We define the matrix $M[\mathbf{X}]$ by
      \begin{equation}\label{neweq12}
      M[\mathbf{X}]_{kj} \ = \   P(X_k > X_j) - \frac{1}{2}.
      \end{equation}
      So that
      \begin{equation}\label{neweq13}
      M[\mathbf{X}]  \ = \ \frac{1}{4} M[\mathbf{f}].
      \end{equation}

      We thus immediately obtain from Theorem \ref{newtheo10}

      \begin{cor}\label{newcor11} The set $\{  M[\mathbf{X}] \}$, with $\mathbf{X}$ varying over all $n$-tuples of
      independent, proper, continuous random variables on $[0,1]$, is  star-shaped about $0$ in the vector space
of real anti-symmetric $n \times n$ matrices. \end{cor}
\vspace{.5cm}

We conclude with some constructions.

For $\mathbf{f} = (f_1, \dots, f_n), \mathbf{g} = (g_1, \dots, g_n) \in \bar \G^n$, $h \in \bar \G$  and $\pi$ a permutation on $[n]$ we define
\begin{align}\label{inteq17}
\begin{split}
\mathbf{f^*}  \ = \ &(f^*_1, \dots, f^*_n), \\
\mathbf{g}\odot \mathbf{f}  \ = \ &(g_1 \odot f_1, \dots, g_n \odot f_n), \\
\r_h\mathbf{f} \ = \ &(\r_h(f_1), \dots, \r_h(f_n)), \\
\pi\mathbf{f} \ = \ &(f_{\pi^{-1}1}, \dots, f_{\pi^{-1}n}).
\end{split}
\end{align}
We call $\mathbf{f}\odot \mathbf{f}$ the \emph{double} of $\mathbf{f}$.

Notice that the operations $\cdot^*, \odot$ and $\pi$ leave $\G_0^n$ invariant while $\r_h$ need not.

For a tournament $R$ we define $\pi R = \{ (\pi i, \pi j) : (i,j) \in R \}$. \index{$\pi R$}

\begin{prop}\label{genericprop6} Let $\mathbf{f} = (f_1, \dots, f_n) \in  \G_0^n$.

(a) The $[n]$-tuple $\mathbf{f^*} \in  \G_0^n$ with $R[\mathbf{f^*}]$
the reversed digraph $R[\mathbf{f}]^{-1}$. Furthermore,
$\mathbf{f^*}$ is generic (or strongly generic) if $\mathbf{f}$ is generic (resp. strongly generic).

(b) The double $\mathbf{f}\odot \mathbf{f} \in  \G_0^n$ with $R[\mathbf{f}\odot \mathbf{f}]  = R[\mathbf{f}]$,
and $\mathbf{f}\odot \mathbf{f}$ is generic (or strongly generic) if $\mathbf{f}$ is generic (resp. strongly generic).

(c) If   $\mathbf{g} \in (\G_{00})^n$,
then $\mathbf{g}\odot \mathbf{f} \in  \G_0^n$  with $R[\mathbf{g}\odot \mathbf{f} ] = R[\mathbf{f}]$,
and $\mathbf{g}\odot \mathbf{f}$ is generic (or strongly generic) if $\mathbf{f}$ is generic (resp. strongly generic).

(d) If $h \in \G$, then $\r_h(\mathbf{f}) \in  \G^n$  with $R[\r_h(\mathbf{f}) ] = R[\mathbf{f}]$.

(e) If  $\pi$ is a permutation on $[n]$, then  $\pi \mathbf{f} \in \G_0^n$
 with $R[\pi \mathbf{f}] = \pi R[\mathbf{f}]$ and $\pi \mathbf{f}$ is generic (or strongly generic) if $\mathbf{f}$ is generic (resp. strongly generic).
\end{prop}

\begin{proof} (a) Since $(f^*_i)^e = - f^e_i$ it follows that  $\{ (f^*_1)^e, \dots, (f^*_n)^e \}$ is linearly independent
when $\{ (f_1)^e, \dots, (f_n)^e \}$ is.  It is clear that if $\{ i, f_1, \dots, f_n \}$ is linearly independent,
 then $\{ i, f_1^*, \dots, f_n^* \}$ is, since $i^* = i$.

Since $f_k^* \circ (f_j^*)^{-1} = (f_k \circ f_j^{-1} )^*$ the reversal of the signs of the integrals follows from (\ref{inteq1}).

 (b) and (c) If $z_0 i + \sum_{j=1}^n z_j g_j \odot f_j = 0$ then for
 $t \in [0,1],  z_0(2t - 1) + \sum_{j=1}^n z_j  f_j(2t - 1)$ is a constant (Note that $i \odot i = i$). Since
 $i, f_1, \dots, f_n \in \G_0$ the constant is zero. Hence,  $z_0 i + \sum_{j=1}^n z_j  f_j = 0$ on $[-1,1]$. That is, for arbitrary $\mathbf{g}$,
 $\mathbf{g}\odot \mathbf{f} \in \mathcal{L}\mathcal{I}\mathcal{N}_n $ if $\mathbf{f} \in \mathcal{L}\mathcal{I}\mathcal{N}_n$.

(b) By (\ref{inteq15})  $$(f_j\odot f_j)^e(t)  =  \frac{1}{4}[ f_j(2t  - 1) - f^*_j(2t - 1)] = \frac{1}{2}f_j^e(2t  - 1)$$
 for $t \in [0,1]$. Consequently,  $\mathbf{f}\odot \mathbf{f} \in \mathcal{L}\mathcal{I}\mathcal{N}_n $ if
 $\mathbf{f} \in \mathcal{L}\mathcal{I}\mathcal{N}_n$.

 It follows from Proposition \ref{elemgroupprop2}(i) and
 (\ref{inteq2})  that
\begin{equation}\label{inteq18}
 Q((f_k \odot f_k),(f_j \odot f_j)) \ = \ \frac{1}{2} Q(f_k,f_j).
 \end{equation}
So $R[\mathbf{f}\odot \mathbf{f} ] = R[\mathbf{f}]$ and  $\mathbf{f}\odot \mathbf{f}$ is strongly generic if $\mathbf{f}$ is.

(c) By (\ref{inteq15})  $(g_j\odot f_j)^e(t)  =  \frac{1}{4}[ f_j(2t  - 1) - g_j(2t - 1)]$ for $t \in [0,1]$.

If $\sum_j x_j f_j \in \G_{00}$ then $\sum_j x_j f^e_j = 0$. If $\mathbf{f} $ is strongly generic and $\mathbf{g} \in (\G_{00})^n$
it follows that  $\mathbf{g}\odot \mathbf{f} \in \mathcal{L}\mathcal{I}\mathcal{N}_n^+$.
Moreover, since $\G_{00}$ is a subgroup
contained in $\G_0$, it follows from Proposition \ref{elemgroupprop2}(i) and (\ref{inteq2})  that
\begin{equation}\label{inteq19}
 Q((g_k \odot f_k),(g_j \odot f_j)) \ = \ \frac{1}{4} Q(f_k,f_j).
\end{equation}

So $R[\mathbf{g}\odot \mathbf{f} ] = R[\mathbf{f}]$ and $\mathbf{g}\odot \mathbf{f}$ is strongly generic if $\mathbf{f}$ is.

(d) Apply (\ref{coneq4a}).

(e) It is clear that $\pi \mathbf{f}$ is generic (or strongly generic) if $\mathbf{f}$ is.
Observe that $(\pi \mathbf{f})_i \to (\pi \mathbf{f})_j$ means
 $f_{\pi^{-1}i} \to f_{\pi^{-1}j}$ and
so $(\pi^{-1}i, \pi^{-1}j) \in  R[\mathbf{f}]$. This is equivalent to $(i,j) \in \pi R[\mathbf{f}]$.

\end{proof}
\vspace{.5cm}

\begin{prop}\label{genericprop7} Let $\mathbf{f} = (f_1, \dots, f_n)$ be a strongly generic $n$-tuple, with $R[\mathbf{f}] = R$ the associated
tournament on $[n]$. Assume that $\mathbf{g} = (g_1, \dots, g_n) \in \G_{00}^n$ is such that for all $(i,j) \in R$, $g_i \to f_j$ and
$f_i \to g_j$. If for $(x_1, \dots, x_n) \in (0,1]^n$, we let $h_i = x_i f_i + (1 - x_i)g_i$, then
$\mathbf{h} = (h_1, \dots, h_n)$ is a strongly generic $n$-tuple, with $R[\mathbf{h}] = R$. \end{prop}

\begin{proof} Notice that $g_j \in \G_{00}$ has been chosen so that for all $i \in [n]$ with $i \not= j$,
$Q(f_i,g_j) > 0$ when $Q(f_i,f_j) > 0$ and $Q(f_i,g_j) < 0$ when $Q(f_i,f_j) < 0$. By Corollary  \ref{fttheo3} such $g_j$ always exist.

Because $h_j^e = x_j f_j^e$ it follows that $\mathbf{h} \in \mathcal{L}\mathcal{I}\mathcal{N}_n^+ $.
Furthermore, because $Q$ is affine in each variable $Q(h_i,h_j)$ is equal to
\begin{equation}\label{inteq20}
 x_i x_j Q(f_i,f_j) + x_i(1 - x_j)Q(f_i,g_j) + (1 - x_i)x_jQ(g_i,f_j).
\end{equation}
All of these terms have the same sign as $Q(f_i,f_j)$. So $\mathbf{h}$ is  strongly generic  with $R[\mathbf{h}] = R$.

\end{proof}

\vspace{1cm}

\section{Universal Tournaments}
\vspace{.5cm}

In this section, we will consider infinite as well as finite tournaments. We will write $(S,R)$ for a tournament $R$ on a set $S$ or just use $R$ when
$S$ is understood.
For a set $S$ we write $|S|$ for the cardinality of $S$.

Let $(S_1,R_1)$ and $(S_2,R_2)$ be tournaments. A \emph{tournament morphism} \index{tournament morphism} $\phi : R_1 \tto R_2$ is
a mapping $\phi : S_1 \tto S_2$ such that $(\phi \times \phi)^{-1}(R_2) \subset R_1$. That is, $\phi(i) \to \phi(j)$ in $R_2$ implies
$i \to j$ in $R_1$. Because $R_1$ and $R_2$ are tournaments, $i \to j$ in $R_1$ implies $\phi(i) \to \phi(j)$ in $R_2$ unless $\phi(i) = \phi(j)$.
An injective morphism is called an \emph{embedding} \index{tournament embedding}in which case $i \to j$ in $R_1$
if and only if $\phi(i) \to \phi(j)$ in $R_2$.
A bijective morphism is called an \emph{isomorphism}, \index{tournament isomorphism} in which case, the inverse map
$\phi^{-1} : S_2 \tto S_1$ defines the
inverse isomorphism $\phi^{-1} : R_2 \tto R_1$. An isomorphism from $R$ to itself is called an
\emph{automorphism} of $R$. \index{tournament automorphism}

Recall that if $(S,R)$ is a tournament and $S_1 \subset S$, then the
tournament on $S_1$  $R|S_1 = R \cap (S_1 \times S_1)$ is
called the \emph{restriction} of $R$ to $S_1$.
The inclusion map $inc : S_1 \tto S$ defines an embedding of $R|S_1$
into $R$.  On the other hand, if $\phi : R_1 \tto R_2$ is
an embedding and $S_3 = \phi(S_1) \subset S_2$, then $\phi : S_1 \to S_3$
defines an isomorphism of $R_1$ onto the restriction
$R_2|S_3$.

For tournaments $(S,R)$ and $(T,U)$ if $S_0 \subset S$ and $\phi : R|S_0$ $ \tto U$ is an embedding, we say that
$\phi$ \emph{extends} to $R$ if there exists an embedding $\psi : R \tto U$ such that $\psi = \phi$ on $S_0$.

\begin{df}\label{unidef00} If $(T,U)$ is a tournament and $T_0 \subset T$,
then we say that $T_0$ satisfies the \emph{simple extension property}
\index{simple extension property} \index{extension property!simple}in
$U$ if for every subset $J \subset T_0$ there exists $v_J \in T$ such that in $U$
\begin{equation}\label{unieq00}
v_J \to j \quad \text{for all} \ j \in J, \qquad \text{and} \qquad j \to v_J  \quad \text{for all} \ j \in T_0 \setminus J.
\end{equation}
We will describe this by saying $v_J$ \emph{chooses} $J \subset T_0$
for $U$. \index{$v_J$ chooses $J \subset T_0$ for $U$}
\end{df}
\vspace{.5cm}

Since a tournament contains no diagonal pairs, it follows that the
 set $\{ v_J : J \subset T_0 \}$ consists of $2^{|T_0|}$ vertices
disjoint from $T_0$. Hence, $|T| \geq |T_0 | + 2^{|T_0|}$.

\begin{lem}\label{unilem00a} (a) Assume $(S,R)$ is a tournament, and
$S_0 \subset S$ with $|S \setminus S_0| = 1$, i.e.
$S$ contains a single additional vertex. If $\phi : R|S_0 \tto U$ is a
tournament embedding and $\phi(S_0)$ satisfies the simple
extension property in $U$, then $\phi$ extends to $R$.

(b) If $(S_0,R_0)$ is a tournament, then there exists a tournament $R$ on a  set $S$ with $S_0 \subset S$ and
$R_0 = R|S_0$ such that $S_0$ satisfies the simple extension property in $R$. Furthermore, $|S| = |S_0| + 2^{|S_0|}$.
\end{lem}

\begin{proof} (a) If $S = \{ v \} \cup S_0$, then we let $J = \phi(R(v)) \subset \phi(S_0)$
and we obtain the extension by mapping
$v$ to $v_J$.

(b) If $*$ is a point not in $S_0$ and $P(S_0)$ is the power set of $S_0$, the
we let $S = S_0 \cup (\{ * \} \times P(S_0))$.
Let $R$ be a tournament on $S$ which contains $R_0$ and such that
\begin{equation}\label{unieq01}
\{ ((*,J),j) : \ j \in J \} \ \cup \ \{(j,(*,J)) :   \ j \in S_0 \setminus J \} \ \subset \ R.
\end{equation}
Clearly, $v_J = (*,J)$ chooses $J \subset S_0$ for $R$ and so  $S_0$ satisfies the simple extension property.

\end{proof}

\begin{df}\label{unidef01} A tournament $(T,U)$  is called \emph{ universal} when it satisfies \index{extension property}
\index{universal tournament}\index{tournament!universal}the following.
\vspace{.25cm}

{\bfseries Extension Property} If $R$ is a  tournament on a countable set $S$, $S_0$ is a finite subset of $S$ and
$\phi : R|S_0 \tto U$ is an embedding, then $\phi$ extends to $R$.
\end{df}
\vspace{.5cm}

\begin{prop}\label{uniprop02} In order that a tournament $(T,U)$  be universal, it is necessary and sufficient that
every finite subset $T_0$ of $T$ satisfies the simple extension property in $U$. \end{prop}

\begin{proof} If $U$ is universal and $T_0$ is a finite subset of $T$, then by Lemma \ref{unilem00a} (b) there exists
a finite set $T_1 \supset T_0$  and a tournament $R_1 \supset U|T_0 $ such that $T_0$ satisfies the simple extension property
in $R_1$.  Let $\psi : R_1 \to U$ be an extension of the inclusion of $U|S_0$ into $U$. If $J \subset T_0$ and $u_J \in T_1$ chooses
$J \subset T_0$ for $R_1$, then
 $v_J = \psi(u_J)$ chooses $J \subset T_0$ for $U$.

Now assume that every finite subset of $T$ satisfies the simple extension property in $U$.

Count the finite or countably infinite set of vertices $v_1, v_2, \dots $ of $S \setminus S_0$.
Let $S_k = S_0 \cup \{ v_1, \dots, v_k \}$. Inductively, with $\psi_0 = \phi$,  Lemma \ref{unilem00a}(a) implies that
we can define an embedding $\psi_k :  R|S_k \tto U$
which extends $\psi_{k-1}$ for $k \geq 1$. If $S$ is finite with
$|S \setminus S_0| = N$ then $\psi = \psi_N$ is the required
extension.  If $S$ is countably infinite then $\psi = \bigcup_k \psi_k$,
 with $\psi(v_i) = \psi_k(v_i) $ for all $k \geq i$ is the
required extension.

\end{proof} \vspace{.5cm}

\begin{prop}\label{uniprop03} Assume that $\{ (T_k,U_k) : k \in \N \}$ is an
increasing sequence of tournaments, i.e. $T_k \subset T_{k+1}$
and  $U_k = U_{k+1}|T_k$ for and $k \in \N$. If $T_k$ satisfies the simple extension property
in $U_{k+1}$ for all $k \in \N$ then $U = \bigcup_k U_k$ is a universal tournament on $T = \bigcup_k T_k$. \end{prop}

\begin{proof} It is clear that the union $U$ is a tournament on $T$. Also,
$|T_{k+1}| \geq |T_k| + 2^{|T_k|}$ and so $T$ is infinite.

If $S_0$ is a finite subset of $T$ and $J \subset S$, then there exists
$k \in \N$ such that $S_0 \subset T_k$ and so $J \subset T_k$.
If $v_J \in T_{k+1}$ chooses $J \subset T_k$ for $U_{k+1}$, then it chooses $J \subset S_0$ for $U$. Thus,
$S_0$ has the simple extension property in $U$. Hence, $U$ is universal by Proposition \ref{uniprop02}.

\end{proof} \vspace{.5cm}

\begin{theo}\label{unitheo04} Assume that $(T_1,U_1)$ and $(T_2,U_2)$ are
countable, universal tournaments. If $S$ is a finite subset of $T_1$ and
$\phi : U_1|S \tto U_2$ is an embedding, then $\phi$ extends to an
isomorphism  $\psi : U_1 \tto U_2$. \end{theo}

\begin{proof}  This is a standard back and forth argument. First note
that  a universal tournament contains copies of every finite tournament and
so must be infinite. Let $u_1,u_2, \dots $ be a counting of the vertices
of $T_1 \setminus S_0$ and $v_1, v_2, \dots $ be
a counting of the vertices of $T_2 \setminus \bar S_0$ with $S_0 = S$ and
 $\bar S_0 = \phi(S)$.  Let $\psi_0: U_1|S_0 \tto U_2|\bar S_0$ be the isomorphism
obtained by restricting $\phi$.
%

Inductively, we construct for $k \geq 1$  \begin{itemize}
\item $S_k \supset S_{k-1} \cup \{ u_k \},$
\item $ \bar S_k \supset \bar S_{k-1} \cup \{  v_k \}, $
\item  $\psi_k : U_1|S_k \tto U_2|\bar S_k$ an isomorphism which extends $\psi_{k-1}$.
\end{itemize}
 Define $S_{k+.5}$ to be $S_k$ together with the
first vertex of $T_1 \setminus S_0$ which is not in $S_k$ and extend
$\psi_k$ to define an embedding $\psi_{k + .5}$ on $S_{k + .5}$.
Let $\bar S_{k + .5} = \psi_{k + .5}(S_{k + .5})$ so that
 $\psi_{k + .5}^{-1} : U_2|\bar S_{k + .5} \tto U_1$ is an embedding.
Define $\bar S_{k+1} $ to be $\bar S_{k +.5}$ together with the first  vertex of $T_2 \setminus \bar S_0$
which is not in it. Extend to define the embedding
$\psi_{k+1}^{-1}: U_2|\bar S_{k+1} \tto U_1$ and let $S_{k+1} = \psi_{k+1}^{-1}(\bar S_{k+1})$.

The union $\psi = \bigcup_k \psi_k$ is the required isomorphism.

\end{proof} \vspace{.5cm}

\begin{cor}\label{unicor05} There exist  countable universal tournaments, unique up to isomorphism.  In fact, if
$(T_1,U_1)$ and $(T_2,U_2)$ are countable, universal tournaments with $i_1 \in T_1, i_2 \in T_2$ then there
exists an isomorphism $\psi : U_1 \tto U_2$ with $\psi(i_1) = i_2$.

Any countable tournament can be
embedded in any universal tournament. \end{cor}

\begin{proof} Beginning with an arbitrary finite tournament  we can use Lemma \ref{unilem00a} (b) to
construct  inductively a sequence of finite tournaments to which Proposition \ref{uniprop03}
 applies, thus obtaining a countable universal tournament.

Since the restriction $U_1|\{ i_1 \}$ is empty, the
map $i_1 \mapsto i_2$ gives an embedding of $U_1|\{ i_1 \}$ into $U_2$. It extends to an isomorphism by
Theorem \ref{unitheo04}.

If $(S,R)$ is a countable tournament with $i_1 \in S$, then, as above, the map taking
$i_1$ to any point of $i_2 \in T_2$ is an
embedding of $R|\{ i_1 \}$ which extends to an embedding of $R$ into $U_2$.

\end{proof} \vspace{.5cm}

\begin{cor}\label{unicor06} Let $(T,U)$ be a universal tournament and $S$ be a finite subset of  $T$.

(a) If $i_1, i_2 \in T$, then there exists an automorphism $\psi$ of $U$ with $\psi(i_1) = i_2$.

(b) If $\phi$ is an automorphism of $U|S$, then there exists an automorphism $\psi$ of $U$ which
restricts to $\phi$ on $S$.
\end{cor}

\begin{proof} (a) This follows from Corollary \ref{unicor05}.

(b) Since the composition of $\phi$ with the inclusion of $S$ is an embedding,
(b) follows from Theorem \ref{unitheo04}.

\end{proof} \vspace{.5cm}

Let $(T,U)$ be a tournament and $S$ be a nonempty, finite subset of $T$. For $J \subset S$, let
  \begin{align}\label{unieq02}
  \begin{split}
T_J = \{ i \in T : (i,j) \in U \ &\text{for all} \ j \in J  \\
\text{and} \ (j,i) \in U \ &\text{for all} \ j \in S \setminus J \}.
\end{split}
\end{align}
That is, $T_J$ is the set of $i \in T$ which choose $J \subset S$ for $U$.

Clearly, $\{ S \} \cup \{ T_J : J \subset S \}$ is a partition of $T$ into $1 + 2^{|S|}$ subsets.

\begin{prop}\label{uniprop07}  If $(T,U)$ is a universal tournament, $S$ is a finite
subset of  $T$ and $J $ is a subset of $S$,
then the restriction $(T_J,U|T_J)$ is a universal tournament. \end{prop}

\begin{proof} Assume that $S_1$ is a finite subset of $T_J$ and $K \subset S_1$. Let
 $S_1^+ = S_1 \cup S$ and $K^+ = K \cup J$. Because $K^+$
satisfies the simple extension property in $U$, there
exists $v_{K^+} \in T$ such that $v_{K^+} $ chooses $K^+ \subset S^+$ for $U$.
It follows first that $v_{K^+} $ chooses $J \subset S$ for $U$ and so
$v_{K^+}\in T_J$. It then follows that   $v_{K^+} $ chooses $K \subset S_1$ for $U|T_J$. Thus, $S_1$ satisfies
the simple extension property in $U|T_J$. As $S_1$ was arbitrary, it follows from Proposition \ref{uniprop02}
that $U|T_J$ is universal.

\end{proof} \vspace{.5cm}

\begin{ex}\label{uniex08} (a) For $(T,U)$ a countable, universal tournament, there
 exists $T_0$ a proper infinite subset of $T$
and an embedding of $U|T_0$ into $U$ which cannot be extended to an embedding of $U$ into itself.

(b) There exists a tournament which is not universal but into which every countable tournament can be embedded.
\end{ex}

\begin{proof} Let $i \in T$, $J = S = \{ i \}$, $T_0 = T_J$ and $T_1 = S \cup T_J$.
 Since $U|T_0$ is universal by Proposition \ref{uniprop07},
Corollary \ref{unicor05} implies that there exists an isomorphism $\phi : U|T_0 \tto U$. Since $\phi$ is surjective,
it cannot be extended to an embedding even of $U|T_1$ into $U$.

Since $(T_0,U|T_0)$ is universal and $T_0 \subset T_1$, it follows
that every countable tournament can be embedded into $U|T_1$.
Let $i_1 \in T_{\emptyset}$ so that $i \to i_1$ in $U$. The inclusion of $i$ into $T_1$ cannot be extended to
an embedding of $ U|\{i,i_1 \}$ into $U|T_1$ since $j \to i$ for
every $j \not= i$ in $T_1$. So $U|T_1$ is not universal.

\end{proof} \vspace{.5cm}

We apply all this to the digraph $\Gamma_{\G}$ from the previous sections.

\begin{theo}\label{unitheo09}  Assume that $\mathbf{f} = (f_1, \dots, f_n)$ is a strongly generic
$n$-tuple in $\G_0^n$. The finite sequence $\{ f_1, \dots, f_n \}$ can be extended
to  an infinite sequence $\{ f_1, f_2, \dots \}$ in $\G_0$,
such that the sequence of even functions
$\{ f_1^e, f_2^e, \dots \}$ is linearly independent and the restriction of
$\Gamma_{\G}$ to the set $\{ f_1, f_2, \dots \}$ is a
universal tournament. \end{theo}

\begin{proof}  Let $U$ be a universal tournament on $\N$. The
tournament $R[\mathbf{f}]$ on $[n]$ can be embedded in $U$. By
permuting $\N$, we may assume that the embedding is given by the
inclusion of $[n]$.  That is, so that $R[\mathbf{f}] = U|[n]$.

Inductively apply Theorem \ref{generictheo4}(c) to construct for $k > n$ the sequence
$\{f_1, \dots, f_k \}$ such that $\mathbf{f_k} = (f_1, \dots, f_k) \in \G\mathcal{E}\mathcal{N}_k^+$
and such that $R[\mathbf{f_k}] =  U|[k]$.

Since  every finite subset is linear independent,  $\{ f_1^e, f_2^e, \dots \}$ is linearly independent.
Finally, $k \mapsto f_k$ induces an isomorphism from $U$ to the restriction $\Gamma_{\G}|\{ f_1, f_2, \dots \}$.

\end{proof} \vspace{.5cm}

For a function $ \mathbf{f} : \N \tto \G_0$, i.e. an element of $\G_0^{\N}$, we define the associated
digraph $U[\mathbf{f}]$ \index{$U[\mathbf{f}]$} on $\N$
\index{tournament associated to $\mathbf{f}$} \index{tournament!associated}
to be the pullback of $\Gamma_{\G}$.  That is,
\begin{equation}\label{unieq09}
i \to j \ \text{in} \ U[\mathbf{f}] \quad \Longleftrightarrow \quad f_i \to f_j \ \text{in} \ \Gamma_{\G}.
\end{equation}

\begin{theo}\label{unitheo10} The sets
  \begin{align}\label{unieq10}
  \begin{split}
  \mathcal{U}\mathcal{T}\mathcal{O}\mathcal{U}\mathcal{R} &=
 \{  \mathbf{f}  \in \G_0^{\N} : U[\mathbf{f}]  \ \text{ is a universal tournament}  \} \\
  \mathcal{U}\mathcal{G}\mathcal{E}\mathcal{N} =
   \{ &\mathbf{f} \in \mathcal{U}\mathcal{T}\mathcal{O}\mathcal{U}\mathcal{R} :
  \{ i, f_1, f_2, \dots \} \ \text{is linearly independent} \} \\
  \mathcal{U}\mathcal{G}\mathcal{E}\mathcal{N}^+ =
   \{ &\mathbf{f} \in \mathcal{U}\mathcal{T}\mathcal{O}\mathcal{U}\mathcal{R} :
  \{ f_1^e, f_2^e, \dots \} \ \text{is linearly independent} \}
  \end{split}
  \end{align}
  \index{$\mathcal{U}\mathcal{T}\mathcal{O}\mathcal{U}\mathcal{R}$}
   \index{$\mathcal{U}\mathcal{G}\mathcal{E}\mathcal{N}$}
  \index{$\mathcal{U}\mathcal{G}\mathcal{E}\mathcal{N}^+$}
are dense $G_{\d}$
subsets of  $\G_0^{\N}$.  \end{theo}

\begin{proof} The projection map from $ \G_0^{\N}$ to $\G_0^n$ is open map for every
$n \in \N$. It follows that the preimage of
a dense $G_{\d}$ set is a dense $G_{\d}$ set. Thus, the condition on $\mathbf{f}  \in \G_0^{\N}$ that
$(f_1,\dots,f_n) \in \mathcal{T}\mathcal{O}\mathcal{U}\mathcal{R}_n $ for every $n$, is a dense $G_{\d}$
condition by Corollary \ref{ftcor2} and the Baire Category Theorem.  Similarly,
by Proposition \ref{genericprop3} the conditions $(f_1,\dots,f_n) \in \mathcal{G}\mathcal{E}\mathcal{N}_n $
for every $n$ and  $(f_1,\dots,f_n) \in \mathcal{G}\mathcal{E}\mathcal{N}^+_n $
for every $n$,, are dense $G_{\d}$ conditions. The first  condition says that the digraph
$U[\mathbf{f}]$ is a tournament.
The latter  conditions say, in addition, that $\{i, f_1, f_2, \dots \}$
or $\{ f_1^e, f_2^e, \dots \}$ is linearly independent.

Given two disjoint finite subsets $J_1, J_2$ of $\N$, let
  \begin{align}\label{unieq11}
  \begin{split}
W(J_1,J_2) = \{ \mathbf{f}  \in \G_0^{\N} \ &: \text{there exists} \ k \in \N \\
\text{such that} \ (k,j_1), (j_2,k) \in \ &U[\mathbf{f}] \ \text{for all} \ j_1 \in J_1, j_2 \in J_2 \}.
   \end{split}
  \end{align}
  Because $f_k \to f_j$ is an open condition, it follows that $W(J_1,J_2)$ is an open subset.
  Intersecting over all disjoint pairs $J_1,J_2$
  we obtain the $G_{\d}$ set $W$.
  From Lemma \ref{uniprop02} it follows that
 $\mathbf{f} \in \mathcal{U}\mathcal{T}\mathcal{O}\mathcal{U}\mathcal{R}$ if and only if $\mathbf{f} \in W$ and
 $(f_1,\dots,f_n) \in \mathcal{T}\mathcal{O}\mathcal{U}\mathcal{R}_n$ for all $n$.  Furthermore,
 $\mathbf{f} \in \mathcal{U}\mathcal{G}\mathcal{E}\mathcal{N}^+$ if and only if $\mathbf{f} \in W$ and
 $(f_1,\dots,f_n) \in \mathcal{G}\mathcal{E}\mathcal{N}^+_n$ for all $n$.
 Similarly, for $\mathcal{U}\mathcal{G}\mathcal{E}\mathcal{N}$.
  So these are all $G_{\d}$ conditions.

 For density, it suffices to show that $\mathcal{U}\mathcal{G}\mathcal{E}\mathcal{N}^+$ is dense.
 To begin with we may perturb $\mathbf{f}$ to get the dense condition that
  $(f_1,\dots,f_n) \in \mathcal{G}\mathcal{E}\mathcal{N}^+_n$ for all $n$.
 Let $N$ be arbitrarily large. We can apply Theorem \ref{unitheo09} to obtain
 $\mathbf{f'} \in \mathcal{U}\mathcal{G}\mathcal{E}\mathcal{N}^+$
 with $f'_k = f_k$ for $k = 1, \dots, N$. By definition of the product
 topology, choosing $N$ large enough we obtain $\mathbf{f'}$ arbitrarily
 close to $\mathbf{f}$.

\end{proof} \vspace{.5cm}

\begin{theo}\label{unitheo11} If $(\N,U)$ is a universal tournament, then there
exists a sequence $\{ X_1, X_2, \dots \}$ of independent,
continuous random variables on $[0,1]$ with $E(X_i) = \frac{1}{2}$ for all $i$,
such that $P(X_i > X_j) > \frac{1}{2}$ if and only if
$i \to j$ in $U$. \end{theo}

\begin{proof} By Theorem \ref{unitheo09} and uniqueness of the universal tournament, we may choose
$\mathbf{f} \in \mathcal{U}\mathcal{T}\mathcal{O}\mathcal{U}\mathcal{R} $ such that $U[\mathbf{f}] = U$.
For each $i \in \N$, let
$  F_i \in \H_0 $ equal $ A_q^{-1}(f_i) = q \circ f_i \circ q^{-1}$ with
$q(t) = \frac{t + 1}{2}$. Let $\{ Z_1, Z_2, \dots \}$ be a sequence of
independent $Unif(0,1)$ random variables and let $X_i = F_i^{-1}(Z_i)$.
From (\ref{02}) and (\ref{inteq}) it follows that
 $P(X_i > X_j) > \frac{1}{2}$ if and only if $\int_{-1}^1 f_j(f_i^{-1}(t)) \ dt > 0$ and so if and only if
$i \to j$ in $U$.

\end{proof}

We conclude this section with a sketch of the measure version of the density result Theorem \ref{unitheo10}.

Let $\M$ \index{$\M$} denote the space of Borel measures on $[0,1]$. With the
weak$^*$ topology induced from the dual space of $\CC([0.1])$, $\M$
becomes a compact metrizable space. A continuous map $G : [0,1] \tto [0,1]$ induces
the continuous map $G_* : \M \tto \M$ by
$G_* \mu(A) = \mu(G^{-1}(A))$ for $A$ a measurable subset of $[0,1]$. We call
$\mu \in \M$ a \emph{proper measure} \index{proper measure} when
the \emph{mean value}\index{mean value} $\int_0^1 x \mu(dx) = \frac{1}{2}$. We let $\M_0$ \index{$\M_0$} denote the
subset of proper measures. For example, the Lebesgue measure
$\lambda$ \index{$\lambda$} on $[0,1]$ is proper.

We call $\mu$ a \emph{continuous measure} \index{continuous measure} when it is full and
nonatomic, i.e. every point has measure zero and
every nonempty open subset of $[0,1]$ has positive measure. We let $\M^c$\index{$\M^c$}
denote the set of continuous measures with
$\M^c_0 = \M^c \cap \M_0$ \index{$\M^c_0$}.

\begin{prop}\label{uniprop12} The set $\M^c$ is a dense $G_{\d}$ subset of $\M$ and $\M_0^c$ is
a dense $G_{\d}$ subset of $\M_0$.
\end{prop}

\begin{proof} By Fubini's Theorem, the measure $\mu$ is nonatomic if and only if the diagonal
$\Delta \subset [0,1]\times [0,1]$ has $\mu \times \mu$ measure
zero.  This holds if and only if for every $\ep > 0$ there exists a non-negative function
$h$ on $[0,1] \times [0,1]$ which $= 1$ on
$\Delta$ but whose $ \mu \times \mu$ integral is less than $\ep$. Thus,
 $(\mu \times \mu)(\Delta) = 0$ is a $G_{\d}$ condition.

For a continuous non-negative function $h \in \CC([0,1])$, the condition
 $\int_0^1 h(x) \mu(dx) > 0$ is an open condition.
Intersecting with a suitable countable collection of functions $h$ we see that having full
support is a $G_{\d}$ condition as well.

If $\mu$ has an atom at $0$ or $1$, we replace $\mu$ by $G_*\mu$ with $G(x) = \ep + (1 - 2\ep)x$. If $\mu$
has mean $\frac{1}{2}$, then
$$\int_0^1 x G_*\mu(dx) = \int_0^1 G(x) \mu(dx) = \ep + (1 - 2\ep)\cdot \frac{1}{2} = \frac{1}{2}.$$
With small $\ep > 0$  the new measure is close to $\mu$ and has no atom at $0$ or $1$.

We may replace an atom at $a \in (0,1)$ by a distribution with the same weight, uniform
on $(a-\ep,a+\ep)$. With small $\ep > 0$ the new measure is
arbitrarily close to $\mu$ and has the same mean.

Finally, with small $\ep > 0$ the measure $\ep \lambda + (1 - \ep)\mu$ is full and has mean $\frac{1}{2}$ if $\mu$ does.

\end{proof}

For $\mu \in \M$ the distribution function $F_{\mu}$ on $[0,1]$ is defined by $F_{\mu}(x) = \mu([0,x))$ for $x \in [0,1]$.
It is clear that $\mu$ is a continuous measure if and only if
$F_{\mu} \in \H$. If $F \in \H$, then $\mu = (F^{-1})_* \lambda$ is the continuous measure with $F_{\mu} = F$ because
$\mu([0,x)) =\lambda((F^{-1})^{-1}[0,x)) = \lambda([0,F(x))) = F(x)$. The map $L : \H \to \M^c$ given by
$L(F) = (F^{-1})_* \lambda$ is a continuous bijection which maps $\H_0$ onto $\M^c_0$.

We can identify ${\mathbf \mu} \in (\M)^{\N}$ with the product measure $\mu_1 \times \mu_2 \times \dots$ on the
product space
$[0,1]^{\N}$. Given $\mathbf{\mu} \in (\M)^{\N}$ we define the associated digraph
$U[\mathbf{\mu}]$ \index{$U[\mathbf{\mu}]$}
on $\N$ by \index{tournament associated to $\mathbf{\mu}$} \index{tournament!associated}
  \begin{equation}\label{unieq12}
  (i,j)  \in U[\mathbf{\mu}] \quad \Longleftrightarrow \quad \mathbf{\mu}(\{ \mathbf{x} \in [0,1]^{\N} : x_i > x_j \}) \ > \ \frac{1}{2}.
  \end{equation}

  Define $\mathcal{U}\mathcal{T}\mathcal{O}\mathcal{U}\mathcal{R}\mathcal{M}$ \index{$\mathcal{U}\mathcal{T}\mathcal{O}\mathcal{U}\mathcal{R}\mathcal{M}$}
  to be the set of $\mathbf{\mu} \in (\M_0^c)^{\N}$
  such that  $U[\mathbf{\mu}]$ is a universal tournament.

  \begin{theo}\label{unitheo13} The set $\mathcal{U}\mathcal{T}\mathcal{O}\mathcal{U}\mathcal{R}\mathcal{M}$
  is a dense $G_{\d}$ subset of  $\M_0^{\N}$. \end{theo}

  \begin{proof} The condition $(i,j) $ or $(j,i) \in U[\mathbf{\mu}]$ is an open condition on $\mathbf{\mu}$ and so the condition that
  $U[\mathbf{\mu}]$ be a tournament is a $G_{\d}$ condition. Defining $W(J_1,J_2)$ as in (\ref{unieq11}) and proceeding as in the proof of
  Theorem \ref{unitheo10} we see that $\mathcal{U}\mathcal{T}\mathcal{O}\mathcal{U}\mathcal{R}\mathcal{M}$ is a $G_{\d}$ set.

  For $\mathbf{f} \in \G^{\N}$ we define $\mathbf{F} = A_q^{-1}(\mathbf{f}) \in \H^{\N}$ letting $F_i = A_q^{-1}(f_i)$ as in the proof of
  Theorem \ref{unitheo11}. It follows that $\bar L = L \circ A_q^{-1}$ is a continuous bijection from
  $\G^{\N}$ to $(\M^c)^{\N}$ which maps $\G_0^{\N}$ onto $(\M^c_0)^{\N}$, dense in $\M_0^{\N}$.

  The map $\bar L$ maps $\mathcal{U}\mathcal{T}\mathcal{O}\mathcal{U}\mathcal{R}$ onto $\mathcal{U}\mathcal{T}\mathcal{O}\mathcal{U}\mathcal{R}\mathcal{M}$.
  By Theorem \ref{unitheo10} $\mathcal{U}\mathcal{T}\mathcal{O}\mathcal{U}\mathcal{R}$ is dense in $(\G_0)^{\N}$.
  It follows that $\mathcal{U}\mathcal{T}\mathcal{O}\mathcal{U}\mathcal{R}\mathcal{M}$ is dense in $(\M^c_0)^{\N}$ and so  in $\M_0^{\N}$.

  \end{proof}

\vspace{1cm}

 \section{Partition Tournaments}
\vspace{.5cm}

An $n$ \emph{partition} \index{partition}$\A = \{ A_1, \dots A_n \}$ of $[Nn] = 1, \dots, Nn$ consists of $n$ disjoint subsets with union $[Nn]$. We call it a
\emph{regular $n$ partition} \index{regular $n$ partition} when the cardinality $|A_i| = N$ for $i = 1, \dots n$.
There are $(nN)!/(N!)^n$ regular $n$ partitions of $[Nn]$.

We define for a regular $n$ partition on $[Nn]$ the digraph
\begin{equation}\label{parteq01}
R[\A] \ = \ \{ (i,j) \in [n] \times [n]: |\{ (a,b) \in A_i \times A_j : a > b \}| > N^2/2 \}.
\end{equation}
\index{$R[\A]$}That is, $(i,j) \in R[\A]$ or $A_i \to A_j$ if it is more likely that a randomly chosen element of $A_i$ is greater than a randomly chosen element of $A_j$
than the reverse. \index{tournament associated to $\A$} \index{tournament!associated}

If $N$ is  odd, then $R[\A]$ is a tournament on $[n]$.  That is, for every pair $i, j \in [n]$ with $i \not= j$ either
$A_i \to A_j$ or $A_j \to A_i$ and not both. Note that for $i = j$, $|\{ (a,b) \in A_i \times A_i : a > b \}| = N(N-1)/2$.

We can think of the partition as the values on the faces on $n$ different $N$-sided dice, but now with values selected from $[Nn]$, and with
 the $Nn$ different faces all having different values. If $D_i$ is the random variable associated with the die having faces with values from $A_i$, then
$A_i \to A_j$ exactly when $D_i \to D_j$ in the previous sense.

If we repeat each label $n$ times then we obtain $n$ different $Nn$ sided dice with labels from $[Nn]$, i.e. $Nn$-sided dice in the sense of
Section \ref{sec1}. However, the dice are only proper when the sum of the members of each $A_i \in \A$ is $\frac{1}{2}N(Nn + 1)$ or, equivalently,
if the expected value of a random choice from $A_i$ is $\frac{1}{2}(Nn + 1)$.

For example, from (\ref{exeq1}) we see that
\begin{align}\label{exeq2}
 \begin{split}
A_1 \ &=  \ \{  3, 5, 7 \}, \\
A_2 \ &=  \ \{ 2, 4, 9 \}, \\
A_3 \ &=  \ \{  1, 6, 8 \}.
 \end{split}
 \end{align}
is a regular $3$-partition of $[9]$ with $A_1 \to A_2 \to A_3 \to A_1$.

For $N$ large enough we can obtain any tournament on $[n]$ by using a regular $n$ partition on $[Nn]$.

\begin{theo}\label{maintheopart} If $R$ is a tournament on $[n]$, then there is a positive integer $M$ such that for every integer $N \geq M$,
there exists a regular $n$ partition of $[Nn]$
$\A = \{ A_1, \dots, A_n \}$
such that for $i, j \in [n]$, $A_i \to A_j$ if and only if $i \to j$ in $R$.  That is, $R = R[\A]$.
\end{theo}
\vspace{.5cm}

\begin{proof} From Theorem \ref{maintheo2} we can choose $X_1, \dots, X_n$  independent, continuous random variables on $[0,1]$ so that
for some $\ep > 0$, and all $(i,j) \in R$, $P(X_i > X_j) > \frac{1}{2} + \ep$. Let $N$ be an integer greater than $1$.

Now for $i \in [n]$ and $\a \in [N]$ let $\{ X_i^{\a} \}$ be independent random variables with each $X_i^{\a}$ distributed like
$X_i$.  We think of $\{ X_i^{\a} : \a \in [N] \}$ as $N$ points, independently chosen in $[0,1]$ according to the distribution of
$X_i$. Since the random variables are continuous, the probability that $X_i^{\a} = X_j^{\b}$ equals zero unless $i = j$ and $\a = \b$.
Thus, with probability one $\{ X_i^{\a} \}$ consists of $nN$ distinct points in $[0,1]$.

Define the indicator function $k : [0,1] \times [0,1] \to \{ 0, 1 \}$ with $k(x,y) = 1$ if $x > y$ and $= 0$ otherwise.
Hence, $k(X_i,X_j)$ is a Bernoulli random variable
with expectation $P(X_i > X_j)$ which is greater than $ \frac{1}{2} + \ep$ if $(i,j) \in R$.

For each $i \not = j$
\begin{equation}\label{parteq02}
\sum_{\a, \b} \ k(X_i^{\a},X_j^{\b}) \ = \ |\{ (\a,\b) \in [N] \times [N] : X_i^{\a} > X_j^{\b} \}|.
\end{equation}

Consider the random variable
\begin{align}\label{parteq03}
\begin{split}
Z_{i,j}  \ = \ [\frac{1}{N^2}\sum_{\a, \b} \ &k(X_i^{\a},X_j^{\b})] - P(X_i > X_j) \ = \ \frac{1}{N^2}\sum_{\a, \b} \ Z_{i,j}^{\a,\b}, \\
\text{with} \quad Z_{i,j}^{\a,\b} \ &= \ k(X_i^{\a},X_j^{\b}) - E(k(X_i^{\a},X_j^{\b})).
\end{split}
\end{align}

Thus, the expectation $E(Z_{i,j}) = 0$. To compute the variance $ = E(Z_{i,j}^2)$, we recall
that the variance of a Bernoulli random variable and the covariance
of two Bernoulli random variables are each bounded by $\frac{1}{4}$, since $p(1-p)$ has its maximum at $p = \frac{1}{2}$.
\begin{equation}\label{parteq04}
E(Z_{i,j}^2) \ = \ \frac{1}{N^4}\sum_{\a_1, \a_2, \b_1, \b_2} \ E( Z_{i,j}^{\a_1,\b_1} \cdot Z_{i,j}^{\a_2,\b_2})
\end{equation}
There are $N^2$ terms with $\a_1 = \a_2$ and $\b_1 = \b_2$, $N^2(N-1)$ terms with $\a_1 = \a_2$ and $\b_1 \not= \b_2$ and
$N^2(N-1)$ terms with $\a_1 \not= \a_2$ and $\b_1 = \b_2$. Each of the terms is bounded by $\frac{1}{4}$. The remaining terms
are all zero by independence.  It follows that the variance of $Z_{i,j}$ is bounded by $\frac{1}{2N}$.

By Chebyshev's Inequality \index{Chebyshev's Inequality} (see, e.g. \cite{BH} Theorem 10.1.11)
\begin{equation}\label{parteq05}
P(|Z_{i,j}| > \ep )\  \leq  \ \frac{1}{2N \ep^2}.
\end{equation}
In $R$ there are
$n(n-1)/2$ pairs $(i,j)$. Hence, the probability that $P(|Z_{i,j}| > \ep )$ for some pair $(i,j) \in R$ is bounded by $\frac{n^2}{4N \ep^2}$.

If $N > \frac{n^2}{4\ep^2}$, then there is positive probability such that, for all $(i,j) \in R, \ $ $|Z_{i,j}| \leq \ep $. In that case
for every $(i,j) \in R$
\begin{equation}\label{parteq06}
\frac{1}{N^2}\sum_{\a, \b} \ k(X_i^{\a},X_j^{\b}) \ > \ \frac{1}{2}.
\end{equation}

Thus, when $N > \frac{n^2}{4\ep^2}$, there exist $X_i^{\a} \in [0,1]$ distinct and so that $|Z_{ij}| \leq \ep$ for every pair $(i,j) \in R$. Let
 $a_1 < \dots < a_{nN}$ list the values of the $X_i^{\a}$'s  in order.  Thus, $a_p < a_q$ if and only if $p < q$. Let
 \begin{equation}\label{parteq07}
 A_i \ = \ \{ p : a_p = X_i^{\a} \quad \text{for some} \ \a \in [N] \}.
 \end{equation}
  From (\ref{parteq02}) and (\ref{parteq06}) it follows that, for all $(i,j) \in R, \ $
 $A_i \to A_j$.

\end{proof}
\vspace{.5cm}

Because there are only finitely many tournaments on $[n]$, we can, as before, choose $M$ so that if $N$ is greater than $M$, then
every tournament on $[n]$ occurs as the tournament of a regular $n$ partition of $[nN]$.

In \cite{A} the label \emph{game} \index{game} is used for a \emph{regular tournament}\index{regular tournament}\index{tournament!regular}.
A digraph is regular when the number of outputs equals the number of inputs for every vertex. Thus, a tournament on $p$ vertices is regular when
$p$ is odd and each vertex has $\frac{p-1}{2}$ inputs and so $\frac{p-1}{2}$ outputs. Up to isomorphism there is a unique regular
tournament on five vertices.  Such a tournament models the extension of the Rock-Paper-Scissors game to
  Rock-Paper-Scissors-Lizard-Spock as was popularized on the television show \emph{The Big Bang Theory}.

\begin{ex}\label{exSac1} An explicit example of a regular $5$ partition on $[30]$ which mimics the regular tournament on five vertices. \end{ex}

\begin{proof}  My student, Julia Saccamano, constructed the following lovely example.
 \begin{align}\label{parteq08}
 \begin{split}
 A \ &= \ \{  1, 6, 10, 22, 24, 30 \}, \\
 B \ &= \ \{ 7, 12, 13, 15, 19, 27 \}, \\
 C \ &= \ \{  3, 4, 17, 18, 23, 28 \}, \\
 D \ &= \ \{  2, 9, 11, 16, 26, 29 \}, \\
 E \ &= \ \{  5, 8, 14, 20, 21, 25 \}.
 \end{split}
 \end{align}
  \begin{equation}\label{parteq09}
  A \to C, E; \ B \to A, D; \ C \to B, D; D \to A, E; E \to B, C.
  \end{equation}
  Furthermore, in each case the victorious probability, i.e. $P(A > C),$ $ P(A > E)$, etc is $\frac{19}{36}$.
In addition, the sum of the faces for each die is $93$ and so the expected value of a roll for each is $\frac{93}{6} = \frac{31}{2}$.

  \end{proof}
\vspace{1cm}

\section{Appendix: An Alternative Proof}
\vspace{.5cm}

In this section we present an alternative proof of Theorem \ref{maintheo3}.

\begin{df}\label{specialdf} We call a sequence $\{ p_0, p_1, \dots \}$ of nonzero,
 continuously differentiable
elements of $\CC([-1,1])$ a \emph{special sequence}\index{special sequence}
when it satisfies the following properties:

\begin{itemize}
\item (Even) All the $p_i$'s are even functions ($p_i(-t) = p_i(t)$ for all $t \in [-1,+1]$) and $p_0 = 1$.

\item (Orthogonal) For all $i \not= j, $
$$ \int_{-1}^{1} \ p_i(t) \cdot p_j(t) dt = 2 \int_{0}^{1} \ p_i(t) \cdot p_j(t) dt = 0.$$

\item (Bounded) $|| p_i || \leq \frac{1}{2}$ for all $i > 0$.

\item (Boundary Values) $p_i(\pm 1) = 0$  for all $i > 0$.
\end{itemize}
\end{df}

To construct an example, recall that that the Legendre polynomials \index{Legendre polynomials} $\{ \ell_n: n = 0, 1, \dots \}$
define an orthogonal sequence on $[-1,+1]$ with $\ell_0 = 1$
and $\ell_n(1) = 1$ for all $n$.  In addition, they consist of only
even power nonzero terms when $n$ is even and so define even functions for even $n$.
Thus, we obtain a special sequence by choosing $p_0 = 1$ and for $i > 0, \ p_i = C_i [\ell_{4i} - \ell_{4i - 2}]$
with the positive constant $C_i$ chosen small
enough to obtain the boundedness condition.

Given a special sequence $\{ p_0, p_1, \dots \}$,  we define the associated sequence $\{g_1, g_2, \dots \} $ by
 \begin{equation}\label{assoceq}
 g_i(t) \ = \ \int_0^t \ [1 + p_i(s)] ds \qquad \text{for} \ t \in [-1,1].
 \end{equation}
 Observe that $g_i$ is odd because $1 + p_i$ is even. Since $g_i' = 1 + p_i \geq \frac{1}{2}$,
 the function $g_i$ is increasing.
 Since $p_i$ is orthogonal to $1$, $g_i(1) = 1$. Thus, each $g_i \in \G_{00}$.  Since
 $ \frac{3}{2} \geq g_i' \geq \frac{1}{2}$ it follows that
$ 2 \geq (g_i^{-1})' \geq \frac{2}{3}$.

Define $\CC_0^1 $ to be the set of $ \xi \in \CC $ such that
\begin{itemize}
\item $\xi(\pm 1) = 0.$
\item $\int_{-1}^{1} \ \xi(t) \ dt \ = \ 0, $
\item $\xi$ is continuously differentiable.
\end{itemize}

For $\xi \in \CC_0^1$ and $|z|$ sufficiently small in $\R$, $H_i(\xi,z) = g_i^{-1} + z \xi$ has a  derivative with
$1/3 < H_i(\xi,z)'(t) < 3$ for $t \in [-1,1]$. In addition,
 $H_i(\xi,z)(\pm 1) = \pm 1$. Thus each such $H_i(\xi,z) \in \G_0$.

 Observe first that
 \begin{equation}\label{assoceqaa}
 \int_{-1}^{1} \ g_j(H_i(\xi, 0)(t)) \ dt \ = \ \int_{-1}^{1} \ g_j(g_i^{-1}(t)) \ dt \ = \ 0,
 \end{equation}
 because $g_j \circ g_i^{-1} \in \G_{00}$.

 Furthermore,
\begin{align}\label{epzeroeq}
\begin{split}
\frac{d}{d z}|_{z= 0} \ \int_{-1}^{1} \ g_j(H_i(\xi, z)(t)) \ dt \ &= \
\int_{-1}^{1} \ g_j'(g_i^{-1}(t))\cdot \xi(t) \ dt \\
= \ \int_{-1}^{1} \ g_j'(s) g_i'(s) \xi(g_i(s)) \ ds \ &= \ \int_{-1}^{1} \ [1 + p_j(s)] \eta(s) \ ds,
\end{split}
\end{align}
with
\begin{align}\label{etaeq}
\begin{split}
\eta(s) \ = \  g_i'(s) \xi(g_i(s))  &\quad \text{and so} \\
\xi(t) \ = \ \eta(g_i^{-1}(t)) \div g_i'(g_i^{-1}(t)) &\ = \ \eta(g_i^{-1}(t)) \cdot (g_i^{-1})'(t).
\end{split}
\end{align}

Notice that
\begin{equation}\label{etaeq2}
\int_{-1}^{1} \ \eta(s) \ ds \ = \ \int_{-1}^{1}  g_i'(s) \xi(g_i(s)) \ ds
\ = \  \int_{-1}^{1} \ \xi(t) \ dt.
\end{equation}
Hence,
\begin{equation}\label{etaeq3}
\int_{-1}^{1} \ \eta(s) \ ds \ = \ 0 \qquad \Longleftrightarrow \qquad  \int_{-1}^{1} \ \xi(t) \ dt \ = \ 0.
\end{equation}

We now prove the following which implies Theorem \ref{maintheo3}.

\begin{theo}\label{maintheo3a} Let $\{ p_0, p_1, \dots \}$ be a
special sequence in $\CC([-1,1])$ with associated sequence $\{ g_1, g_2, \dots \}$ in $\G_{00}$.
If $R$ is a tournament on $[n]$ and $\ep > 0$, then there exists an $n$-tuple  ${\mathbf f} = ( f_1,\dots, f_n ) \in \G_0^n $ such that
\begin{itemize}
\item For $i \in [n], \ ||f_i - g_i|| < \ep$.
\item For $i \in [n], t \in [-1,1], \ 1/3 < f_i'(t) < 3.$
\item For $i,j  \in [n],$ with $i \not= j$,
$$f_j \circ f_i^{-1} \in \G_+ \quad \Longleftrightarrow \quad (i,j) \in R.$$
\end{itemize}
\end{theo}

\begin{proof} By induction on $n$. With $n = 1$ the only tournament is empty and we can let $f_1 = g_1$.

For the inductive step, assume that $n > 1$ and define
 \begin{align}\label{etadefeq}
\begin{split}
\eta(s) \ = \ \sum_{j \in R(n)} \ p_j(s) \ - &\ \sum_{j \in R^{-1}(n)} \ p_j(s) \\
\xi(t) \ = \ \eta(g_n^{-1}(t)) \div g_n'(g_n^{-1}(t)) &\ = \ \eta(g_n^{-1}(t)) \cdot (g_n^{-1})'(t).
\end{split}
\end{align}
From (\ref{etaeq3}) it follows that that $\xi \in \CC_0^1$.

So for $j = 1, \dots, n-1$ the orthogonality assumptions imply
\begin{equation}\label{esteq}
\int_{-1}^{1} \ [1 + p_j(s)] \eta(s) \ ds  =
\begin{cases} \int_{-1}^{1}   p_j(s)^2  ds > 0 \ \text{for} \ j \in R(n), \\
- \int_{-1}^{1}   p_j(s)^2  ds < 0 \ \text{for} \ j \in R^{-1}(n). \end{cases}
\end{equation}

There exists $\ep_1$ with $\ep > \ep_1> 0$ so that with $|z| < \ep_1$,
$H_n(\xi,z) \in \G_0$ with derivative between $1/3$ and $3$ on $[-1,1]$.
From (\ref{assoceqaa}) and (\ref{epzeroeq}) it follows that we can choose $z$ with
$0 < z < \ep_1$ so that with $f_n^{-1} = H_n(\xi,z)$ we have

\begin{equation}\label{esteq2}
\int_{-1}^{1} \ g_j(f_n^{-1}(t)) \ dt \
\begin{cases} > 0 \ \text{for} \ j \in R(n), \\
< 0 \ \text{for} \ j \in R^{-1}(n).\end{cases}
\end{equation}

That is,
\begin{equation}\label{esteq3}
g_j(f_n^{-1}) \in \G_+ \ \text{for} \ j \in R(n), \ \text{and} \ g_j(f_n^{-1}) \in \G_- \ \text{for} \ j \in R^{-1}(n).
\end{equation}

Because $\G_{\pm}$ are open sets, there exists $\d$ with
$0 < \d < \ep$ such that $||f_j - g_j|| < \d$ for $j \in [n-1]$ implies
\begin{equation}\label{esteq4}
f_j(f_n^{-1}) \in \G_+ \ \text{for} \ j \in R(n), \ \text{and} \ f_j(f_n^{-1}) \in \G_- \ \text{for} \ j \in R^{-1}(n).
\end{equation}

We apply the induction hypothesis to the restricted tournament $\bar R = R|[n-1]$ and we
obtain $f_j$ with $||f_j - g_j|| < \d$ for $j \in [n-1]$ such that for $j, k \in [n-1], f_k \circ f_j^{-1} \in \G_+ $
if and only if $(j,k) \in \bar R$.

From (\ref{esteq4}) it follows that $\{f_1, \dots, f_{n-1}, f_n \}$ is the required list.

\end{proof}

\vspace{1cm}

\bibliographystyle{amsplain}

\printindex
\end{document}